\documentclass[a4paper,11pt]{article}
\usepackage{epsf,epsfig}
\usepackage{graphics,color}
\usepackage{amsmath}
\usepackage{amssymb}
\usepackage{cite}
\usepackage{verbatim}
\usepackage{float}
\usepackage{graphicx}
\usepackage{amsthm}
\usepackage{textcomp}
\usepackage{subfig}
\usepackage{enumitem}
\usepackage{bigints}
\usepackage{esint}

\newtheorem{theorem}{Theorem}[section]

\newtheorem{cor}[theorem]{Corollary}
\newtheorem{lemma}[theorem]{Lemma}

\newtheorem{remark}[theorem]{Remark}
\newtheorem{definition}[theorem]{Definition}

\numberwithin{equation}{section}
\definecolor{Aquamarine}{cmyk}{0.82,0,0.30,0}

\newcommand{\black}{\color{black}}

\newcommand{\jump}[1]{\text{{\rm \textlbrackdbl}}{#1}\text{{\rm \textrbrackdbl}}}

\def\Borelset{B}
\def\openBrai{\Omega}
\def\subopenBrai{A}
\def\newconst{\overline C}
\def\constcompu{H}

\def\constg{C}

\def\ATeps{{\rm F}_\var}

\def\ATtwoeps{{{\rm F}_\var^{\Suno}}}
\def\ATtwoepsk{{{\rm F}_{\var_k}^{\Suno}}}

\def\circle{{\mathbb S}^1}

\def\constlemmabase{C}

\def\dimension{n}
\def\dki{d_k^{(i)}}

\def\domAT{\mathcal D}

\def\domATpiccolo{\domAT_{\Suno}}
\def\etaprimo{\eta'}
\def\etasecondo{\eta''}

\def\funzioni_vito_k{\phi_k}

\def\Hausdorff{\mathcal H^{\dimension-1}}

\def\indiceteo{m}
\def\limk{\lim_{k \to +\infty}}
\def\liminfk{\liminf_{k \to +\infty}}
\def\limsupk{\limsup_{k \to +\infty}}

\def\livellok{t(k)}

\def\minvalueu2{m_2[u]}
\def\MMepsk{{\rm MM}_{\var_k}}
\def\ModicaMortolaeps{{\rm MM}_\var}

\def\MS{{\rm F}}
\def\MSlift{{\rm F}_{{\rm lift}}}

\def\NN{{\mathbb N}}
\def\nuovoliftingk{\phi_k}
\newcommand{\Om}{\Omega}

\def\openAT{\Omega}
\def\opensetAT{\openAT}

\def\openset{U}
\newcommand{\pip}{(i)}

\newcommand{\R}{\mathbb{R}}
\def\Rn{{\mathbb R}^\dimension}

\def\sottoaperto{A}
\def\suk{(u_k)}
\def\Suno{\circle}

\def\trunkphi{\varphi}
\def\var{\varepsilon}

\newcommand{\dx}{\, {\rm d}x}

\newcommand{\ds}{\, {\rm d}s}
\newcommand{\dt}{\, {\rm d}t}

\newcommand{\res}      {\mathop{\hbox{\vrule height 7pt width .5pt depth
			0pt\vrule height .5pt width 6pt depth 0pt}}\nolimits}

\usepackage{hyperref}

\title{$\Gamma$-convergence of free discontinuity problems for
circle-valued maps in the linear regime
}
\author{Giovanni Bellettini\footnote{Dipartimento di Scienze Matematiche, Informatiche e Fisiche,
via delle Scienze 206, 33100 Udine UD, Italy, and ICTP International Centre for Theoretical Physics,
Mathematics Section,
Strada Costiera 11, I-34151 Trieste Italy. E-mail: giovanni.bellettini@uniud.it
}\and
	Roberta Marziani\footnote{Dipartimento di  Matematica, Universit\`a degli studi di Trento, 38123 Povo (TN), Italy.
	E-mail: roberta.marziani@unitn.it}
	\and
	Riccardo Scala\footnote{Dipartimento di Ingegneria dell'Informazione e Scienze Matematiche, Universit\`a di Siena, 53100 Siena, Italy.
		E-mail: riccardo.scala@unisi.it}
}

\begin{document}
	\maketitle
	
	\begin{abstract}
We investigate the $\Gamma$-convergence of Ambrosio–Tortorelli-type functionals for circle-valued functions,
in the case  
of volume terms with linear growth.
We show the emergence of a nonlocal $\Gamma$-limit, which is due
to the topological structure of the target space, and discuss compactness
of minimal liftings.
Our results extend the analysis of a previous work on
the quadratic case. Applications to the area functional are given.
\end{abstract}
	

	\vspace{2mm}
	
	\noindent {\bf AMS (MOS) subject clas\-si\-fi\-ca\-tion:} 
	49Q15, 49Q20, 49J45, 58E12.
	\section{Introduction}
	Let $n \geq 1$ and 
$\Omega \subset \mathbb{R}^n$ be a connected and simply connected bounded open set with Lipschitz boundary, and let $\mathbb{S}^1 
=\{(x,y)\in\R^2\colon x^2+y^2=1\}$ denote the unit circle. For $u \colon \Omega \to \mathbb{S}^1$ and $v \colon \Omega \to [0,1]$, 
consider the family of functionals defined, for $\var > 0$, by
	\begin{equation*}\label{ATS1}
	{\rm AT}_\var^{\mathbb{S}^1}(u,v):=\int_\Omega \left(v^2|\nabla u|^2+\var|\nabla v|^2+ \frac{(1-v)^2}{4\var}
	\right)  \dx\,.
	\end{equation*}
As in the classical Ambrosio-Tortorelli approximation \cite{AT90,AT92}, 
one observes a 
decoupling, as $\var\to0^+$, 
between the bulk term $\int_\Omega v^2|\nabla u|^2\dx$ and the phase-field 
term $\int_\Omega(\var|\nabla v|^2+(1-v)^2/\var)\dx$.
However,	
in sharp contrast with the classical setting, 
\cite{BMS25} shows that for $\mathbb{S}^1$-valued maps 
the $\Gamma$-limit of ${\rm AT}_\var^{\mathbb{S}^1}$  
depends on the choice of the functional domain. 
More precisely, if one considers
$\{(u,v)\in W^{1,1}(\Omega;\mathbb{S}^1)\times W^{1,2}(\Omega) \ 
\text{with}~ v|\nabla u|\in L^2(\Omega)\}$, 
as  domain,
	then $	{\rm AT}_\var^{\mathbb{S}^1}$ $\Gamma$-converges in $L^1$ 
to the Mumford-Shah functional for circle-valued maps, namely to
	\begin{equation*}
{\rm MS}_{\mathbb{S}^1}(u):=\int_\Omega|\nabla u|^2\dx+\Hausdorff(S_u)\,, \quad u\in SBV^2(\Omega;\mathbb{S}^1)\,.
	\end{equation*}
If, instead, one restricts to the smaller domain 
$\{(u,v)\in W^{1,2}(\Omega;\mathbb{S}^1)\times W^{1,2}(\Omega)\}$,  
then ${\rm AT}_\var^{\mathbb{S}^1}$ $\Gamma$-converges in $L^1$ to the nonlocal functional
	 \begin{equation*}
	 {\rm MS}_{\rm lift}(u):=\int_\Omega|\nabla u|^2\dx+m_2[u]\,, \quad u\in SBV^2(\Omega;\mathbb{S}^1)\,,
	 \end{equation*}
 where, denoting by $S_\varphi$ the jump set of a function $\varphi$,
 \begin{equation}\label{min-prob}
 	m_2[u]:=\inf\{\Hausdorff(S_\varphi)\colon \varphi\in GSBV^2(\Omega)\,,\ e^{i\varphi}=u \text{ a.e. in }\Omega
 	\}\,.
 \end{equation}
Since $S_u\subseteq S_\varphi$ for any lifting $\varphi$ of $u$, 
it follows ${\rm MS}_{\rm lift}(u)\ge{\rm MS}_{\mathbb{S}^1}(u)$,
possibly with the strict inequality for appropriate maps $u$.

This dichotomy, sometimes 
referred to as the Lavrentiev phenomenon,
 reflects the influence of the topological structure of $\mathbb{S}^1$, 
in particular the possibility of a nonzero degree for circle-valued maps. 
In the first case, one has more freedom in the 
construction of the recovery sequence, and it is 
possible to preserve the degree throughout the approximation procedure. In contrast, in the 
second case, since a map $u \in W^{1,2}(\Omega; \mathbb{S}^1)$ has zero topological degree whenever $\Omega$ is 
simply connected 
(see, e.g., \cite{Brezis-Mironescu}), one is forced to approximate 
maps of nonzero degree by maps of  zero degree, which entails an additional energetic cost. This phenomenon is 
clearly illustrated by the example of the vortex map described in \cite{BMS25} and also occurs 
in other variational problems concerning the relaxation of the area functional in codimension two \cite{BES} 
and the study of ferromagnetic spin systems \cite{Cicalese-Orlando-Ruf}.

The minimisation problem in \eqref{min-prob} can be formulated for any $u\in SBV^p(\Omega;\mathbb{S}^1)$ with
$p>1$ 
by minimizing 
over all liftings $\varphi\in GSBV^p(\Omega)$; we denote the corresponding value by $m_p[u]$. In 
 \cite{BMS25} it is  proven that $m_p[u]$ admits a minimizer which in general  does not belong to $SBV^p(\Omega)$, 
 and that such minimisers are related either to suitably defined optimal transport problems (when $n=2$) or to certain 
solutions of the Plateau problem (when $n>2$).
\\

  \noindent \textbf{Linear growth setting.}
  In the present paper we consider 
a variant where the quadratic term $|\nabla u|^2$ 
is replaced by a term of the form
$f(|\nabla u|)$,
 with  $f$ having linear growth at infinity (see Section \ref{sec:Setting-and-mainresult} for the precise assumptions). 
  Similarly to the  unconstrained case \cite{ABS,AlicandroFocardi}, the linear growth introduces a genuine interaction 
between the bulk term and the phase-field term
  which is absent in the quadratic case.
In addition, the resulting $\Gamma$-limit is defined on the larger space  $BV(\Omega;\mathbb{S}^1)$ and thus has a Cantor part.

More precisely, by allowing also for more general phase-field terms (cf.~Section \ref{sec:Setting-and-mainresult}), 
our main result
Theorem \ref{thm2} states that if one restricts to the 
domain $(u,v)\in W^{1,2}(\Omega;\mathbb{S}^1)\times W^{1,2}(\Omega)$, then the $\Gamma$-limit is the nonlocal functional 
   \begin{equation*}
 	\int_\Omega f(|\nabla u|)\dx+ |D^cu|(\Omega)+ m_g[u]\,,\quad u\in BV(\Omega;\mathbb{S}^1)\,,
 \end{equation*}
 where $m_g[u]$ is expressed as 
a minimization problem 
associated to the 
cost function $g$ defined in \eqref{def:g}, that is
 \begin{equation}\label{limit2}
 m_g[u]:=\inf\Big\{\int_{S_\varphi}g(|\varphi^+-\varphi^-|)\,{\rm d}\Hausdorff\colon \varphi\in GBV(\Omega)\,,\ e^{i\varphi}=u\text{ a.e. in }\Omega
\Big \}\,.
 \end{equation}
 We also prove that $m_g[u]$ 
is actually a minimum, i.e.,
we prove the  existence of a minimal lifting for $m_g[u]$. 
As a byproduct of this, we establish compactness results for sequences of liftings in the linear-growth setting (Theorems~\ref{thm:compactness} and \ref{teo:compactness}), extending the results of \cite{BMS25} to the full  $BV$-framework.
{In particular, we show that sequences of liftings of a given map in $BV(\Omega;\mathbb{S}^1)$, for which we control the portion of jump set with jump amplitude greater than $\sigma>0$ are precompact with respect to the local convergence modulo $2\pi$ (see Definition \ref{def:convergence}), and moreover, that  the same holds for sequences of liftings corresponding to $u_k\stackrel{*}{\rightharpoonup}u$ in $BV(\Omega;\mathbb{S}^1)$.}

The linear growth case entails substantial differences in the proof strategy. First, compactness results for liftings in the class 
$GBV$ are required, which rely on slicing arguments and measure-theoretic 
localization techniques. Moreover, the genuine interaction between the bulk and the phase-field terms prevents a direct adaptation of the quadratic-growth analysis and necessitates a 
localized approach to 
$\Gamma$-convergence. These tools allow us to simplify several arguments with respect to the quadratic case, while at the same time extending the theory to a broader and more flexible variational framework.
 In addition, our analysis reveals natural connections with optimal transport problems and with the relaxation of the area functional in $\mathbb{S}^1$. We briefly discuss these relations below and defer a complete investigation to future work. \\

  \noindent\textbf{Connections with the relaxation of the area functional in $\mathbb{S}^1$.} 
 The relaxed area in $\mathbb{S}^1$ of the graphs of $L^1$-functions $u\colon\Omega\to \mathbb S^1$ is defined by
 \begin{equation*}\label{relaxed-area}
 	\mathcal{A}^{\mathbb S^1}(u,\Omega):=\inf\left\{\liminf_{k \to +\infty}
 	\int_\Omega\sqrt{1+|\nabla u_k|^2}\dx\colon u_k\in C^1(\Omega;\mathbb{S}^1)\,,\ u_k\to u\text{ in } L^1
 	\right\}\,.
 \end{equation*}
 This functional was characterised in terms of Cartesian currents (cf.~\cite{Giaq}). Alternatively one could propose a characterisation using lifting theory.  For instance, if we further assume that $\Omega$ is simply connected,  then we show that (see Theorem \ref{thm:relaxation-area})
 \begin{equation*}
 	\mathcal{A}^{\mathbb S^1}(u,\Omega)= \int_\Omega\sqrt{1+|\nabla u|^2}\dx+ |D^cu|(\Omega)+ m[u]\,,
 \end{equation*}
 where 
 \begin{equation*}
 	m[u]:=\min\left\{\int_{S_\varphi}|\varphi^+-\varphi^-|{\rm d}\Hausdorff\colon \varphi\in BV(\Omega)\,,\, e^{i\varphi}=u \text{ a.e. in }\Omega
 	\right\}\,.
 \end{equation*} 
 The similarity with \eqref{limit2} (when $f(t)=\sqrt{1+t^2}$) suggests the possibility of approximating 
 $\mathcal{A}^{\mathbb S^1}$ through suitably chosen linear-growth Ambrosio–Tortorelli type functionals. 
Indeed, using a diagonal argument, 
we provide such an approximation in Theorem~\ref{thm:approx-area},
 opening interesting perspectives for future research related to the relaxation of the area functional in codimension two (see \cite{AcDa:94,BEPS,BP,BES,BES2,BES3,BMS,C, Scala, BSS, SS} and references therein for related results). \\

 \noindent
 \textbf{Connections with optimal transport.} The quantity $m_g[u]$ admits a natural interpretation as an optimal transport problem with cost  $g$. Assume for simplicity $\Omega\subset\R^2$ and let 
$u\in W^{1,1}(\Omega;\mathbb{S}^1)$ have distributional 
Jacobian of the form
 \begin{equation*}
 	{\rm Det}(\nabla u)=\pi\sum_{i=1}^{N}(\delta_{x_i}-\delta_{y_i})\,,\quad x_i,y_i\in \Omega\,.
 \end{equation*}
  Setting $\mu:=\sum_{i=1}^N\delta_{x_i}$ and $\nu:=\sum_{i=1}^N\delta_{y_i}$ one can show that
  \begin{equation*}
m_g[u]= \inf 
\int_R g(\theta)\,{\rm d}\mathcal H^1\,,
  \end{equation*}
  where the infimum is taken over all 
integer multiplicity $1$-currents $T=(R,\theta,\tau)$ with $\partial T=\mu-\nu$. We refer to \cite{BMS25} for the discussion on the equivalence between the minimization problem $m_g[u]$ and the optimal transport problem with cost function $g\equiv1$.
   For optimal transport problems in this setting we refer to \cite{MMS,MMT,FDW,Brezis-Mironescu2} and references therein.
   \\

Let us also mention that, due to its property of approximation of the Mumford-Shah type energy of circle-valued maps, as observed in \cite{BMS25}, our Ambrosio-Tortorelli functional has a strict connection with analysis of Ginzburg-Landau energies and dislocation
mechanics (we refer to \cite{BBH} for the former topic). Specifically, the Mumford-Shah functional for $\mathbb S^1$-valued maps has been successfully used to model the appearence of 
dislocations in $2$-dimensional domains (see \cite{DLSVG,CDLS} for more details; we also refer to the introduction of \cite{BMS25} for a more exhaustive discussion on the link with vortex singularities).\\ 

\noindent\textbf{Content of the paper.} 
In Section \ref{sec:notation_and_main_definitions} we collect some 
notation and recall some useful results 
needed to prove our main results. 
In Section \ref{sec:Setting-and-mainresult} we 
state our main result
Theorem \ref{thm2}. 
In Section \ref{sec:existence_of_a_lifting_of_minimal_jump} 
we state two compactness results for sequences of liftings 
(Theorem  \ref{thm:compactness} and Theorem \ref{teo:compactness}). 
We also state and prove existence of a
solution to \eqref{limit2}. 
Section \ref{sec:compactness_of_liftings_of_a_converging_sequence} 
is devoted to the proof of Theorems  \ref{thm:compactness} and \ref{teo:compactness}. 
In Section \ref{sec:Gamma_convergence} we provide the proof
 of Theorem \ref{thm2}. {Eventually, in Section \ref{sec:relaxed-area} we prove Theorems \ref{thm:relaxation-area} and \ref{thm:approx-area}.}

	\section{Notation and preliminaries}\label{sec:notation_and_main_definitions}
		We start
to recall some notions concerning $BV$ and $GBV$ functions \cite{AmFuPa}
and lifting theory \cite{Brezis-Mironescu}. In what follows:
$\partial^*A$ denotes the reduced boundary of a finite perimeter set $A\subset\R^n$,
$\vert \cdot \vert$ and 
 $\mathcal H^{n-1}$ denote the Lebesgue measure and the $(n-1)$-dimensional Hausdorff measure 
in $\R^n$, respectively,
and 
$\chi_A$ denotes 
the characteristic function of the set $A\subset \Rn$.
If $\openBrai\subseteq \Rn$ is an open set
we denote by $\mathcal A(\openBrai)$ the family 
of all open subsets of $\openBrai$.

	{ 
We recall the following 
lemma
\cite[Proposition 1.16]{Braides-Notes}:

\begin{lemma}[Localisation]
\label{lem:localisation-lemma}
	Let $\openBrai\subset \R^n$ be a bounded open set 
with Lipschitz boundary.
		Let $\mu\colon\mathcal A(\openBrai)\to[0,+\infty)$ 
be a superadditive set function on disjoint open sets, 
$\lambda$ be a positive measure on $\openBrai$ and, for $k \in \mathbb N$,
${ h}_k\colon \openBrai\to[0,+\infty]$ be a  
Borel function such that
		\begin{equation*}
\mu(\subopenBrai)\ge \int_{\subopenBrai}
h_k\,{\rm d}\lambda\quad \text{ for every } \subopenBrai 
\in \mathcal{A}(\openBrai) \text{ and } k \in \mathbb N.
		\end{equation*}
	Then
		\begin{equation*}
		\mu(\subopenBrai)\ge \int_\subopenBrai ~
\sup_{k\ge1}h_k \,{\rm d}\lambda\quad \text{ for every } \subopenBrai \in \mathcal{A}(\openBrai)\,.
	\end{equation*}
	\end{lemma}}

	\subsection{$BV$ and $GBV$ functions}\label{subsec:SBV_and_GSBV_functions}
Let $\openset \subset \Rn$ be open and bounded, and $m
\geq 1$ be an integer. 
We denote by $BV(\openset;\R^m)$ the space of 
vector-valued $L^1(U;\R^m)$-\textit{functions with bounded variation in} $\openset$, and 
with $|\cdot|_{BV}$ 
and $\|\cdot\|_{BV}$  
the $BV$ seminorm and norm in $U$, respectively, i.e.
	$$|u|_{BV}:=|Du|(\openset),\qquad \|u\|_{BV}:=\|u\|_{L^1}+|u|_{BV},$$
	see \cite{AmFuPa}.
We recall that if
$u\in BV(\openset;\R^m)$ then 
its distributional gradient is a
$\mathbb{R}^{m\times n}$- valued Radon measure with finite total
variation, and it can be decomposed as
	\begin{equation*}
		Du=\nabla u\mathcal L^n+D^cu+\jump{u}\otimes \nu_u\mathcal{H}^{n-1}\res S_u\,,
	\end{equation*}
	where $\nabla u$ is the approximate gradient of $u$, 
$D^cu$ the Cantor part, $S_u$ is the approximate jump set of $u$, 
$\jump{u}=u^+-u^-$ is the jump opening, $\nu_u$ 
is a unit normal
to the jump set $S_u$ of $u$ 
and $\otimes$ stands for the tensor product, see \cite[Def. 3.67]{AmFuPa}.
  We say that $u\in BV(\openset;\R^m)$  is a special function of bounded variation in $U$ if $D^cu=0$. We denote by $SBV(U;\R^m)$ the space of vector-valued \textit{special functions with bounded variation in }$U$. For $p>1$ we define also the space
\begin{equation*}
SBV^p(\Omega;\R^m)=\{u\in SBV(\Omega;\R^m)\colon |\nabla u|\in L^p(\Omega)\,,\ \Hausdorff(S_u)<+\infty\}\,.
\end{equation*}

A measurable function $u: U \to \R^m$ belongs to the space of \textit{generalised functions with bounded variation in $U$}, that is, $u\in GBV(U;\R^m)$, if $\phi\circ u\in {BV_{\rm loc}(U)}$
for any 
$\phi\in C^1(\R^m)$ with $\nabla \phi$ compactly supported.\footnote{Recall that $f\in BV_{\rm loc}(U)$ if $f\in BV(U')$ for every $U'\subset U$ 
open with $\overline U'$ compact in $U$.} { Note that also in $GBV$ we can define the approximate gradient, the Cantor part and the jump set. Moreover 
$GBV(\Omega,\R^m)\cap L^\infty(\Omega;\R^m)= BV(\Omega,\R^m)\cap L^\infty(\Omega;\R^m)$. Analogously one can define $GSBV(U;\R^m)$ and $GSBV^p(U;\R^m)$.}
If $m=1$ we write $BV(U)=BV(U;\R)$,  and $GBV(U)=GBV(U;\R)$.

We set
			\begin{equation*}
			BV(U;\mathbb{S}^1)=\{u\in BV(U;\R^2)\colon |u|=1\,\text{ a.e. in }\,U\}\,.
		\end{equation*}
Eventually,
a (finite or countable) family $\{E_i\}$ of finite perimeter subsets 
of a finite perimeter set $F\subseteq \Rn$ is called a  Caccioppoli partition of $F$ 
if the sets $E_i$ are pairwise disjoint, their union is $F$, and the sum of their perimeters is finite.
\subsection{Approximation and compactness for 
the vector Mumford-Shah functional}\label{sec:Gamma_approximation_and_compactness_for_the_MS_functional}
We recall a result due to
Alicandro-Braides-Shah \cite{ABS}, subsequently
 extended to the vector case by Alicandro-Focardi \cite{AlicandroFocardi}. 
Let $m \geq 1$, and 
$$
\domAT
:=\{(u,v)\in W^{1,2}(\opensetAT;\R^m)
\times W^{1,2}(\opensetAT)\colon 0\le v\le1\}.
$$
Consider the family of functionals 
$\ATeps
\colon L^1(\Omega;\R^m)\times L^1(\Omega)\to [0,+\infty]$ given by 
	\begin{equation}\label{AT-lin}
	\ATeps(u,v):=\begin{cases}	\displaystyle
		\int_{\opensetAT}\left(\psi(v)f(|\nabla u|)+\varepsilon|\nabla v|^2+\frac{W(v)}{\varepsilon}\right){\rm d}x&\text{ if }(u,v)
		\in\domAT\,,\\[1em]
		+\infty& \text{otherwise in }L^1(\opensetAT;\R^m)
		\times L^1(\opensetAT)\,,
	\end{cases}
\end{equation}
where 

\begin{enumerate}[label=(\alph*)]
	\item\label{hyp:psi} $\psi\colon[0,1]\to[0,1]$ is an 
increasing lower semicontinuous function with $\psi(0)=0$, 
$\psi(1)=1$ and $\psi(t)>0$ if $t\in (0,1)$;
	\item\label{hyp:f} $f\colon[0,+\infty)\to [0,+\infty)$ is 
a convex increasing function  such that 
	\begin{equation*}
		\lim_{t\to +\infty}\frac{f(t)}{t}=1\,;
	\end{equation*}
	\item\label{hyp:W} $W\colon [0,1]\to[0,+\infty)$ is a continuous function such that $W(s)=0$ if and only if $s=1$.
\end{enumerate}
Set  also
\begin{equation}\label{def:cW}
	\quad c_W(t):=2\int_t^1\sqrt{W(s)}\ds\,,
\end{equation}
and 
\begin{equation}\label{def:g}
{	g(z):=\min\{\psi(t)z+2c_W(t)\colon 0\le t\le 1\}\,.}
\end{equation}
Then the following theorem holds.
\begin{theorem}\label{thm:Linear-Ambrosio-Tortorelli} 
Let $\opensetAT
\subseteq\R^n$ be an open set 
with Lipschitz boundary. 
	Then $\ATeps$ 
$\Gamma-L^1$-converge
to the functional $\MS$ defined as 
	\begin{equation*}\label{limite-At-lin}
	\MS(u,v)\!\!:=\!
\!\begin{cases}
		\displaystyle \int_\opensetAT\!\! f(|\nabla u|)\dx
\!+\!|D^cu|(\opensetAT)
\!+\!\!\int_{S_u}\!\!\!g(|u^+-u^-|)\,{\rm d}\Hausdorff & \!{\rm if}~
\!u\!\in\! GBV(\opensetAT;\R^m),  v=1 ~{\rm a.e.}
\\
+\infty & \!{\rm otherwise~\! in~}\! L^1(\opensetAT;\R^m)
\!\times\! L^1(\opensetAT)
\end{cases}
	\end{equation*}
as $\var\to0^+$.
\end{theorem}
\begin{remark}\label{rem:decoupling}\rm
	By inspecting the proof of Theorem \ref{thm:Linear-Ambrosio-Tortorelli} 
one actually deduces the following  properties:
\begin{enumerate}
\item[(i)] {\textbf{Lower bound on open sets.}} 
The lower bound inequality holds on every $A\in \mathcal{A}(\Omega)$. 
Precisely, let $\ATeps(\cdot,\cdot,A)$ and $\MS(\cdot,\cdot,A)$ denote the 
localised functionals on the 
open set $A\subseteq \Om$ and suppose $(u_\var,v_\var)\to(u,v)$ in $L^1(\Omega;\R^m)\times L^1(\Omega)$. Then
\begin{equation*}
\liminf_{\var\to0^+}\ATeps
(u_\var,v_\var,A)\ge \MS(u,v,A);
\end{equation*}
\item[(ii)] {\textbf{$g$ is a truncating function.}} The function $g$ in \eqref{def:g} satisfies the following properties:
\begin{enumerate}[label=$\arabic*.$]
	\item $g$ is increasing, $g(0)=0$ and $\lim_{z\to+\infty}g(z)=2c_W(0)$;
	\item $g$ is subadditive, i.e., $g(z_1+z_2)\le g(z_1)+g(z_2)$ for all $z_1,z_2\in [0,+\infty)$;
	\item $g$ is Lipschitz-continuous with Lipschitz constant 1;
	\item $g(z)\le z$ for all $z\in [0,+\infty)$ and $\lim_{z\to0^+}g(z)/z=1$;
	\item For any $T>0$ there exists $c_T>0$ such that $z\le c_Tg(z)$ for all $z\in [0,T]$;
	\item For $\sigma >0$ there exists $C_\sigma>0$ such that $g(z)\ge C_\sigma$ for all $z\ge\sigma>0$.
\end{enumerate}
	Properties $1$ -- $5$ follow by \cite[Remark 4.2]{ABS}. 
Property $6$ follows 
taking 
$C_\sigma =g(\sigma)$.
\end{enumerate}

\end{remark}

\subsection{Liftings of $\mathbb{S}^1$-valued maps}
\label{subsec:liftings_of_S1_maps}
 Let $\Omega\subset\R^n$ be a bounded connected and simply connected
open  set with Lipschitz boundary, and  $u\colon\Omega\to \mathbb{S}^1$ 
be a measurable function. A \textit{lifting} of $u$ is a measurable function $\varphi\colon\Omega\to\R$ such that 
 \begin{equation*}
 	u(x)=e^{i\varphi(x)}\quad \text{for a.e. }x\in\Omega\,.
 \end{equation*}
{
Given a Borel set $B\subseteq\Omega$, we say that $\varphi:B\rightarrow \R$ is a lifting of $u$ on $B$ if the previous equality holds a.e. on $B$. 
}Clearly, 
if $\varphi$ is a lifting of $u$, then so is $\varphi+2\pi m$ 
for all $m\in \mathbb{Z}$.

If $u$ has some regularity, a natural question is  whether $\varphi$ can be chosen with the same regularity. The answer is 
partially positive, 
see \cite{Brezis-Mironescu,Dav-Ign} for more details: we recall in particular
that if $n\ge2$ and $u\in W^{1,p}(\Omega;\mathbb{S}^1)$ for some $p\in[2,+\infty)$, then, being $\Omega$ simply connected, $u$ has
 a lifting $\varphi\in W^{1,p}(\Omega)$, and $\varphi$ is unique (mod $2\pi$). Moreover,
for $n=2$ and 
$p\in[1,2)$ there exists $u\in W^{1,p}(\Omega;\mathbb{S}^1)$ for which 
there are no liftings in $W^{1,p}(\Omega)$,
see \cite[Theorem 1.2, Remark 1.9]{Brezis-Mironescu}.
Indeed it can be shown
{\cite[pages 17-19]{Brezis-Mironescu}}
that there are no liftings of the vortex map $u_V(x)=x/|x|$ in $W^{1,1}(B_1)$
(and thus there are no liftings in $W^{1,p}(B_1)$ for all $p\in[1,2)$).\\

\noindent
Next we recall the
\begin{theorem}[\textbf{Davila-Ignat}]\label{thm:dav-ign}
	Let $u\in BV(\Omega;\mathbb{S}^1)$. 
Then there exists a lifting $\varphi\in BV(\Omega)$ such that $\|\varphi\|_{L^\infty}\le2\pi$ and $|\varphi|_{BV}\le2|u|_{BV}$.
\end{theorem}
\begin{proof}
See \cite[Theorem 1.1]{Dav-Ign}, and 
also \cite[Theorem 1.4]{Brezis-Mironescu}.
\end{proof}

As usual, for any lifting $\varphi\in SBV(\Om)$ of $u$ 
we write\footnote{$S_\varphi^I$ stands for the ``integer'' part of the
jump,
and $S_\varphi^f$ for the ``fractional'' part.} 
$S_\varphi=S_\varphi^I\cup S_\varphi^f$ where 
$$
S^I_\varphi:=\{x\in S_\varphi:\jump{\varphi}(x)\in 2\pi\mathbb Z\},
\qquad  S_\varphi^f:=S_\varphi\setminus S_\varphi^I.
$$
In particular $S^f_\varphi=S_u$;
{moreover $|\nabla u|=|\nabla \varphi|$ a.e. in $\Omega$. If $u\in BV(\Om;\mathbb S^1)$, also $|D^cu|(\Omega)=|D^c\varphi|(\Omega)$.}

Let $u\in BV(\Omega;\mathbb{S}^1)$ and consider the minimum problem
\begin{equation*}
\inf\{|\varphi|_{BV}\colon \varphi\in BV(\Omega),\, e^{i\varphi}=u\text{ a.e. in }\Omega\}.
\end{equation*}
Then there exists 
a minimizer $\varphi\in BV(\Omega)$ 
such that $|\varphi|_{BV}\le 2| u|_{BV}$ and $0\le \int_\Omega\varphi\dx\le 2\pi|\Omega|$ \cite[page 25]{Brezis-Mironescu}. 
Moreover, as already mentioned in the introduction, in \cite{BMS25} it was proven that for any $p>1$ and any $u\in SBV^p(\Omega;\mathbb{S}^1)$
the minimisation problem 
\begin{equation*}
m_p[u]:=\inf \{\Hausdorff(S_\varphi)\colon \varphi\in GSBV^p(\Omega),\, e^{i\varphi}=u\ \text{a.e. in }\Omega
\}
\end{equation*}
admits a solution.
{Here, we are instead concerned with 
the existence of minimal lifting for  a minimisation problem which depends on the cost function $g$ in \eqref{def:g} (cf.~\eqref{limit2} for its precise definition). 	}
	\section{Statement of the main result} 
\label{sec:Setting-and-mainresult}
Let $n \geq 2$ and  $\Omega\subset\R^n$ be a connected bounded 
open set with Lipschitz boundary.
To the pair $u\colon\Omega\to \mathbb S^1$, $v\colon\Omega\to[0,1]$  
we associate the functional
\begin{equation*}
	{\rm F}_{\var}(u,v):=\int_{\Omega}\left(\psi(v)f(|\nabla u|)+\varepsilon|\nabla v|^2+\frac{W(v)}{\varepsilon}\right){\rm d}x\,,
\end{equation*}
where $\var \in (0,1]$, $\psi$, $f$ and $W$ 
satisfy properties \ref{hyp:psi}, \ref{hyp:f} and \ref{hyp:W} 
in Section \ref{sec:Gamma_approximation_and_compactness_for_the_MS_functional}.

We consider  the  functional domain
given by
\begin{equation}\label{eq:two_domains}
	\begin{aligned}
		\domATpiccolo
		:=\left\{(u,v)\in W^{1,2}(\Omega;\mathbb{S}^1)\times W^{1,2}(\Omega)\colon \ 0\le v\le1
		\right\}\subset 
		L^1(\Om; \circle) \times L^1(\Om)\,.
	\end{aligned}
\end{equation}
For $\var \in (0,1]$
let us consider the corresponding family of 
functionals 
$$
\ATtwoeps\colon L^1(\Omega;\mathbb{S}^1)\times L^1(\Omega)\to[0,+\infty]$$ 
given by
\begin{equation}\label{AT2}
	\ATtwoeps(u,v):=
	\begin{cases}
		{\rm F}_\var(u,v)& \text{ if }(u,v)\in \domATpiccolo\,,\\[1em]
		+\infty& \text{ otherwise in }L^1(\Om; \circle) \times L^1(\Om)\,.
	\end{cases}
\end{equation}

The main result of the present paper reads as follows.
\begin{theorem}[$\Gamma$-convergence of $\ATtwoeps$]\label{thm2}
	{ Provided $\Omega$ is  
simply-connected,} we have 
	$$
	\Gamma-L^1 \lim_{\var\to 0^+} \ATtwoeps = \MSlift,
	$$
	where
	\begin{equation}\label{MSlif}
		\MSlift(u,v):=
\begin{cases}
		\displaystyle	\int_\Omega f(|\nabla u|)\dx+ |D^cu|(\Omega)+
		m_g[u] & 
	{\rm if}~ u\in BV(\Omega;\mathbb{S}^1), v=1 {\rm ~a.e.}, 
\\
	+\infty & {\rm  otherwise~in}
	~L^1(\Om; \circle) \times L^1(\Om),
\end{cases}
	\end{equation}
and $m_g[u]$ is defined in \eqref{limit2}.
In addition, $m_g[u]$ is a  minimum.

\end{theorem}

We 
note that the admissible set in \eqref{limit2} is nonempty, 
according to Theorem \ref{thm:dav-ign}.
As an outcome of our analysis, we derive compactness results for sequences of liftings (see Theorems \ref{thm:compactness} and \ref{teo:compactness}). In particular, Theorem \ref{teo:compactness} will play a key role in establishing the lower bound in Theorem \ref{thm2}.

 \section{On $g$-minimizing liftings}\label{sec:existence_of_a_lifting_of_minimal_jump}
In this section we establish 
the existence of solutions to the minimization problem \eqref{limit2}, for $g$ as in \eqref{def:g} and $u\in BV(\Omega;\mathbb{S}^1)$, as stated in Corollary \ref{cor:existence-minimizer}.
For later use, we introduce the following notion of convergence.
\begin{definition}[\textbf{Local convergence modulo $2\pi$}]\label{def:convergence} 
	 Let $(\varphi_k)_{k\ge1}\subset GBV(\Omega)$ 
and $\varphi_\infty\in GBV(\Omega)$. We say that $(\varphi_k)$ converges
locally modulo $ 2\pi$ to $\varphi_\infty$ (in $\Omega$) if the following holds: 
	there exist a Caccioppoli partition $\{E_i\}_{i\ge1}$ of $\Omega$
	and sequences $(d_k^{(i)})_{k\ge1}\subset\mathbb Z$  for $i\in \mathbb{N}$ such that 
		\begin{equation}\label{eq:compactness}
		\begin{split}
			& \lim_{k \to +\infty} 
			(\varphi_k(x)-2\pi d^{(i)}_k) = \varphi_\infty(x)\qquad 
			\forall i \in \NN, ~ 
			\text{ for a.e. }x\in E_{i}\,,\\
			&\lim_{k \to +\infty} 
			|\varphi_k(x)-2\pi d^{(i)}_k| = +\infty\qquad \forall i \in \NN, ~ 
			\text{ for a.e. }x\in \Om\setminus E_{i}\,.
		\end{split}
	\end{equation}
\end{definition}
\begin{theorem}[\textbf{Compactness}]\label{thm:compactness} { Let $0\le\sigma\le\pi$.}
Let $u\in BV(\Om;\mathbb S^1)$ and  
$(\varphi_k)_{k\ge1}\subset GBV(\Om)$ be a  
sequence of liftings of $u$ in $\Om$  with 
\begin{equation}\label{equi-bound-jump}
\constcompu:= \sup_{k\in\mathbb N}\Hausdorff(S_{\varphi_k}^{\sigma})
<+\infty\,,
\end{equation}
where 
\begin{equation}\label{eq:Ssigma}
S_{\varphi_k}^\sigma:=\{x\in S_{\varphi_k}\colon 
|\jump{\varphi_k(x)}|\ge\sigma\}.
\end{equation}
Then 
there exist
a (not-relabelled) subsequence of indices $k$ and a lifting $\varphi_\infty\in GBV(\Om)$ of $u$ in $\Omega$ such that $(\varphi_k)$ 
converges  locally modulo $2\pi$ to $\varphi_\infty$.
\end{theorem}
A first consequence of Theorem \ref{thm:compactness} is the following result.
\begin{cor}[\textbf{Existence of a minimizer to $m_g[u]$}]\label{cor:existence-minimizer} Let $g$ be as in \eqref{def:g}.
	Let $u\in BV(\Omega;\mathbb{S}^1)$ and let $(\varphi_k)_{k\ge1}\subset GBV(\Omega)$ be a 
sequence of liftings of $u$ in $\Omega$  with 
	\begin{equation}\label{eq:constg}
	\constg:=	\sup_{k\in\mathbb{N}} \int_{S_{\varphi_k}}g(|\varphi^+_k-\varphi^-_k|)\,{\rm d}\Hausdorff<+\infty\,.
	\end{equation}
	Then the following holds:
	\begin{enumerate}
		\item[(i)] {Compactness:}~There exist a (not-relabelled) subsequence of indices $k$ and a lifting
$\varphi_\infty\in GBV(\Omega)$ 
of $u$ in $\Omega$ 
		such that $(\varphi_k)$ converges locally modulo $2\pi$ to $\varphi_\infty$;
		\item[(ii)] {Lower semicontinuity:}
		\begin{equation*}
			\liminf_{k\to+\infty}\int_{S_{\varphi_k}}g(|\varphi^+_k-\varphi^-_k|)\,{\rm d}\Hausdorff\ge \int_{S_{\varphi_\infty}}g(|\varphi^+_\infty-\varphi^-_\infty|)\,{\rm d}\Hausdorff\,.
		\end{equation*}
	\end{enumerate}
	As a consequence, the
infimum defining 
$m_g[u]$
is attained in $GBV(\Omega)$.
\end{cor} 
\begin{proof}
(i).
Let $0<\sigma\le\pi$ be fixed. Remark \ref{rem:decoupling}-(ii) and 
\eqref{eq:constg}
yield
\begin{equation*}
\constg \ge \int_{S_{\varphi_k}}g(|\varphi^+_k-\varphi^-_k|)\,{\rm d}\Hausdorff\ge C_\sigma \Hausdorff(S_{\varphi_k}^\sigma)\,.
\end{equation*}
and the assertion follows from Theorem \ref{thm:compactness}.
\\
(ii). Let
$\varphi_\infty$,  
$\{E_i\}_{i\ge1}$ and 
$(d_k^{(i)})_{k\ge 1}$ be as in \eqref{eq:compactness}.
We exploit the fact that $g$ 
arises from the $\Gamma$-limit of the functionals in \eqref{AT-lin}.
 Fix $N,i\in \mathbb{N}$ and set 
 $$\varphi_k^{i,N}:=((\varphi_k-2\pi d_k^{(i)})\wedge N)\vee (-N)\,.$$
Then $\varphi_k^{i,N}\to\varphi_\infty^{i,N}$ in $L^1(\Omega)$ 
with $\varphi_\infty^{i,N}:=(\varphi_\infty\wedge N)\vee(-N)$ in $E_i$ and 
$\varphi_\infty^{i,N}=\pm N$ in two suitable finite perimeter subsets 
of $\Omega$ the complement (in $\Omega$) of the union of which is $E_i$.

Define 
\begin{equation*}
 G(\phi, A):= \MS(\phi,1,A)\quad \text{ for }\phi \in GBV(\Omega)\,, \ A\in \mathcal{A}(\Omega)
\end{equation*}
where $\MS(\cdot,\cdot,A)$ is  
as in  Remark \ref{rem:decoupling}-(i) with $f(t)=t$.
 {}From Theorem \ref{thm:Linear-Ambrosio-Tortorelli} and properties of $\Gamma$-convergence it follows that $G(\cdot,A)$ is lower semicontinuous in the $L^1$-topology, 
so that 
\begin{equation}\label{lim-inf} 
	\liminf_{k \to +\infty}G(\varphi_k, A)\ge
\liminf_{k \to +\infty}G(\varphi_k^{i,N},A)\ge G(\varphi_\infty^{i,N},A), \qquad
A \in \mathcal A(\Om).
\end{equation}
Next we define the superadditive set function 
\begin{equation}\label{eq:muA}
\mu(A):=\liminf_{k \to +\infty}G(\varphi_k, A) \qquad \forall A \in \mathcal A(\Om),
\end{equation}
and the positive measure 
{\begin{align*}
\lambda(\Borelset):= G(\varphi_\infty, \Borelset)
=&\int_\Borelset |\nabla\varphi_\infty|\dx+ |D^c\varphi_\infty|(\Borelset)
+ \int_{(S_{\varphi_\infty}\setminus \Sigma)\cap \Borelset}g(|\varphi^+_\infty-\varphi^-_\infty|)\,{\rm d}\Hausdorff\,\\&+2c_W(0)\Hausdorff(\Sigma\cap \Borelset),
\end{align*}
}for every Borel set $\Borelset\subseteq\Omega$, 
where we have denoted 
\begin{equation}\label{eq:def_Sigma}
\Sigma:=\displaystyle \Big(\bigcup_{i= 1}^{+\infty}
 \partial^*E_i\Big) \setminus\partial\Om. 
\end{equation}
We identify $E_i$ with the set of points of density $1$ for $\chi_{E_i}$.
Define also, for $i\geq1$, 
{$$
h_1^{i}:=\begin{cases}
	1&\text{in }E_i \setminus (\Sigma\cup S_{\varphi_\infty})
	\\[1em]
	0&\text{otherwise in }\Omega
\end{cases}\,,
$$
and 
$$
h_2^{i}:=\begin{cases}
	1&\text{in } S_ {\varphi_\infty}\cap E_i
	\\[1em]
	0&\text{otherwise in }\Omega
\end{cases}\,, \quad
h_3^{i}:=\begin{cases}
	1&\text{in }  (\partial^* E_i)\setminus \partial\Om
	\\[1em]
	0&\text{otherwise in }\Omega.
\end{cases}\,.
$$
}
By construction we have 
$$\lim_{N\rightarrow +\infty}|(\varphi_\infty^{i,N})^+ -(\varphi_\infty^{i,N})^-|= |\varphi_\infty^+ -\varphi_\infty^-| \qquad \Hausdorff\text{ a.e. in }(S_{\varphi_\infty}\cap E_i),$$	
and  
$$\lim_{N\rightarrow +\infty}|(\varphi_\infty^{i,N})^+ -(\varphi_\infty^{i,N})^-|= +\infty \qquad \Hausdorff\text{ a.e. in }\partial^* E_i,$$
hence by monotonicity of $g$ and the monotone convergence theorem
	we arrive at
$$\lim_{N\rightarrow +\infty}\int_{(S_{\varphi_\infty^{i,N}}\setminus \partial^*E_i) \cap A}g(|(\varphi_\infty^{i,N})^+ -(\varphi_\infty^{i,N})^-|)\,{\rm d}\Hausdorff\ge \int_Ah_2^{i}{\rm d}\lambda 
$$	
and 
$$2c_W(0)\mathcal H^{n-1}(A \cap \partial^*E_i)=
\lim_{N\rightarrow \infty}\int_{A \cap \partial^*E_i}g(|(\varphi_\infty^{i,N})^+ -(\varphi_\infty^{i,N})^-|)\,{\rm d}\Hausdorff\ge \int_Ah_3^{i}{\rm d}\lambda.
$$	
Whence, by \eqref{lim-inf} we conclude
$$\mu(A)\geq \int_Ah_2^{i}{\rm d}\lambda,
\quad \mu(A)\geq \int_Ah_3^{i}{\rm d}\lambda, 
\qquad A \in \mathcal A(\Omega).
$$
On the other hand, still by \eqref{lim-inf} it is easily seen that 
$$\mu(A)\geq \int_Ah_1^{i}{\rm d}\lambda \qquad \forall i \geq 1,\qquad A \in \mathcal A(\Omega).
$$
Therefore,
since
\begin{equation*}
\sup_{j=1,2,3, ~i \in \mathbb N}h_j^{i}\equiv 1\,,
\end{equation*}
using  Lemma \ref{lem:localisation-lemma}  we infer
\begin{equation*}
\mu(A)\ge \lambda(A) \quad \text{for all }A\in \mathcal A({\Omega})\,.
\end{equation*}
As a consequence we get 
\begin{equation}\label{liminf-G}
\liminf_{k \to +\infty}G(\varphi_k,\Omega)\ge G(\varphi_\infty,\Omega)\,.
\end{equation}
Since $\varphi_k,\varphi_\infty$ are liftings of $u$ in 
$\Omega$ it holds $|\nabla\varphi_k|=|\nabla \varphi_\infty|=|\nabla u|$ and $|D^c\varphi_k|(\Omega)=|D^c\varphi_\infty|(\Omega)=|D^c u|(\Omega)$, 
hence \eqref{liminf-G} in turn implies 
\begin{equation*}
\liminf_{k \to +\infty}\int_{S_{\varphi_k}}g(|\varphi_k^+-\varphi_k^-|)\,{\rm d}\Hausdorff \ge 
\int_{S_{\varphi_\infty}}g(|\varphi_\infty^+-\varphi_\infty^-|)\,{\rm d}\Hausdorff\,,
\end{equation*}
and the proof of (ii) is concluded.
\end{proof}
From the proof above we also obtain the following:
\begin{cor}
In the same hypotheses and notations
of Corollary \ref{cor:existence-minimizer},
we have 
	\begin{equation*}
	\liminf_{k\to+\infty}\int_{S_{\varphi_k}}g(|\varphi^+_k-\varphi^-_k|)\,{\rm d}\Hausdorff\ge \int_{S_{\varphi_\infty}\setminus \Sigma}g(|\varphi^+_\infty-\varphi^-_\infty|)\,{\rm d}\Hausdorff\,+2c_W(0)\Hausdorff(\Sigma),
\end{equation*}
where we recall $2c_W(0)=\sup g$ and $\Sigma$ is defined 
in \eqref{eq:def_Sigma}.
\end{cor}

Now we state a generalization of Theorem \ref{thm:compactness}
to a sequence $\suk$,
needed to show Theorem \ref{thm2}; the proof
will be given in the next section.

	\begin{theorem}[\textbf{Generalised compactness}]\label{teo:compactness}
Let $u,u_k\in BV(\Om;\mathbb S^1)$ 
be such that 
\begin{equation}\label{eq:dododo}
u_k\stackrel{*}{\rightharpoonup}u \quad
{\rm in~} BV(\Om;\mathbb S^1),
\end{equation}
and let $\varphi_k\in GBV(\Om)$ be a 
lifting of $u_k$ in $\Om$, for all $k\ge 1$.  
Suppose   that for some $0\le\sigma\le\pi$ 
\begin{equation}\label{equi-bound-jump-uk}
H:= \sup_{k\in\mathbb N}\Hausdorff(S_{\varphi_k}^{\sigma})
<+\infty.
\end{equation}
		Then
there exist a  (not relabelled) subsequence of indices $k$ and  a lifting $\varphi_\infty\in GBV(\Om)$ of $u$ in $\Om$ such that $\varphi_k$ converges locally modulo $2\pi$ to $\varphi_\infty$.
	\end{theorem}

\section{Proofs of the compactness 
Theorems \ref{thm:compactness} and \ref{teo:compactness}} \label{sec:proofs-compactness}
We split the proof of Theorems \ref{thm:compactness} and \ref{teo:compactness} into a number of intermediate lemmas.
\begin{lemma}[\textbf{Localized compactness}]\label{lem:passo_base} 
Let $\Omega\subset\R^\dimension$ be a bounded open set with Lipschitz 
boundary, 
and $u\in BV(\Om;\mathbb S^1)$. 
 {Let $0\le\sigma\le\pi$ and} $(\varphi_k)_{k \geq 1}
\subset GBV(\Om)$ be 
a sequence of liftings of $u$ satisfying
\eqref{equi-bound-jump}.
Let
	$F\subset\Om$ be a nonempty finite perimeter set
in $\Om$. 
	Then, for  a not-relabelled subsequence, 
there exist a sequence 
$(d_k)_{k\ge1}\subset \mathbb{Z}$, a finite perimeter set 
$$
E \subseteq F
$$
in $\Om$, 
 and a function $\varphi_\infty\in GBV(\Om)$, 
such that:
\begin{equation}\label{eq:phi_infty_is_a_lifting}
 \varphi_\infty 
\text{ is a lifting of } u \text{ in } E\,, \quad
\varphi_\infty=0\quad\text{ in }
\Om\setminus E\,,
\end{equation}
and
	\begin{equation}\label{eq:lem5.9}
		\begin{split}
		&	\lim_{k\to+\infty}(\varphi_k(x)-2\pi d_k)= \varphi_\infty(x)\qquad \text{ for a.e. }x\in E\,,
\\
			&  \lim_{k\to+\infty}	|\varphi_k(x)-2\pi d_k|= +\infty\qquad \text{ for a.e. }x\in F\setminus E,
\\
		\end{split}
	\end{equation}
\begin{equation}\label{eq:boh}
\begin{aligned}
& |E|\geq \frac{n^n \omega_n|F|^\dimension}{{2^n}(2
\constcompu
+\Hausdorff(\partial^*F))^\dimension}>0, 
\\
\end{aligned}
\end{equation}
If furthermore
\begin{equation}\label{eq:if_furthermore}
|\varphi_k(x)-2\pi d_k|\rightarrow +\infty \qquad {\rm for~ a.e.~}
x\in \Om\setminus E,
\end{equation}
and {$\{x\in \Om:\varphi_k(x)-2\pi d_k\rightarrow \pm\infty\}$ 
are finite perimeter sets in $\Om$}, 
then 
\begin{align}\label{eq:lem5.9:bis}
 \liminf_{k\rightarrow +\infty}\Hausdorff\big(
S^\sigma_{\varphi_k}\cap A\big)\ge \Hausdorff\big(
A\cap \partial^*E\big)
\quad \text{ for any } A\in \mathcal A( \Om),
\end{align}
and 
\begin{equation}\label{eq:lem5.9:tris}
 \Hausdorff(F\cap\partial^*E)=\Hausdorff(F\cap\partial^* (F\setminus E))\le 
\constlemmabase.
\end{equation}
\end{lemma}
\begin{proof}{Following the lines of the proof of \cite[Lemma 4.1]{BMS25} and noticing that \cite[equation 4.8]{BMS25} still holds if we replace $S_{\varphi_k}$ and $S_{\varphi_1}$ with $S_{\varphi_k}^\sigma$ and $S_{\varphi_1}^\sigma$	we can prove that there exist a sequence $(d_k)_{k\ge1}\subset\mathbb Z$, a measurable set $E\subset F$, and a measurable  function $\varphi_\infty\colon\Omega\to\R$ such that \eqref{eq:phi_infty_is_a_lifting}--\eqref{eq:boh} hold true for a not-relabelled subsequence.}

We next show that $E$ has finite perimeter and $ \varphi_\infty\in GBV(\Omega)$. For each $k\ge1$ we define the auxiliary function 
	\begin{equation*}
		\widehat\varphi_k:=\begin{cases}
\varphi_k-2\pi d_k&\text{ in }F\\[1em]
k&\text{ in } \Omega\setminus F.
		\end{cases}\,.
	\end{equation*}
	Since $F$ has finite perimeter we have $(	\widehat\varphi_k)\subset GBV(\Omega)$. Moreover 
	\begin{equation*}
\sup_{k\in\mathbb{N}}	\Hausdorff(S_{\widehat\varphi_k}^\sigma)\le 
\constcompu + \Hausdorff(\partial^*F)\,,
	\end{equation*}
	and 
		\begin{equation*}\
		\begin{split}
			&	\lim_{k\to+\infty}\widehat\varphi_k(x)= \varphi_\infty(x)\qquad \text{ for a.e. }x\in E\,,
			\\
			&  \lim_{k\to+\infty}	|\widehat\varphi_k(x)|= +\infty\qquad \text{ for a.e. }x\in \Omega\setminus E\,.
		\end{split}
	\end{equation*}
Now we proceed by adapting a
slicing argument in \cite{AT} (see also \cite{CC}).
For $\mathcal H^{n-1}$-a.e. $\xi\in\mathbb{S}^{n-1} \subset \Rn$ 
and $y\in \Pi^\xi:=\{y\in \R^n\colon y\cdot\xi=0\}$ we introduce the following notation:
\begin{equation*}
A_y^\xi:=\{t\in\R\colon y+t\xi\in A\}\,, \quad a_y^\xi(t):=a(y+t\xi)\,,
\end{equation*}
for any $A\subseteq \R^n$ and  
$a\colon A\to\R$. In particular $(\widehat \varphi_k)_y^\xi$ is a map of one variable in $GBV(\Om_y^\xi)$ for a.e. $y\in \Pi^\xi$.

Recalling that 
$S^{\sigma}_{(\widehat\varphi_k)_y^\xi}:=\{t\in S_{(\widehat\varphi_k)_y^\xi}\colon |\jump{(\widehat\varphi_k)_y^\xi(t)}|\ge\sigma\}$ and $S^{\sigma}_{\widehat\varphi_k}:=\{x\in S_{\widehat\varphi_k}\colon |\jump{\widehat\varphi_k(x)}|\ge \sigma\}$, we have 
\begin{equation*}
	S^\sigma_{(\widehat\varphi_k)_y^\xi}\subset (S^\sigma_{\widehat\varphi_k})_y^\xi\,.
\end{equation*}
This together with Fatou's lemma implies
\begin{equation*}
	\begin{split}
		&\int_{\Pi_\xi}\liminf_{k\to+\infty}\left(
		|D(\widehat\varphi_k)_y^\xi|(\Omega_y^\xi\setminus S^{\sigma}_{(\widehat\varphi_k)_y^\xi} )+\mathcal H^0(\Omega_y^\xi\cap S^{\sigma}_{(\widehat\varphi_k)_y^\xi} )\right)
		{\rm d}\Hausdorff(y)\\
		&
\le 
		\int_{\Pi_\xi}\liminf_{k\to+\infty}\left(|D(\widehat\varphi_k)_y^\xi|(\Omega_y^\xi\setminus S^{\sigma}_{(\widehat\varphi_k)_y^\xi} )+\mathcal H^0\left(\Omega_y^\xi\cap (S^{\sigma}_{\widehat\varphi_k} )_y^\xi\right)\right)
		{\rm d}\Hausdorff(y)\\
		&\le \liminf_{k \to +\infty}
		\int_{\Pi_\xi}|D(\widehat\varphi_k)_y^\xi|(\Omega_y^\xi\setminus S^{\sigma}_{(\widehat\varphi_k)_y^\xi} )+\mathcal H^0\left(\Omega_y^\xi\cap (S^{\sigma}_{\widehat\varphi_k})_y^\xi\right)
		{\rm d}\Hausdorff(y)\\
		&\le \liminf_{k \to +\infty}\left( |D\widehat\varphi_k|(\Omega\setminus S^{\sigma}_{\widehat{\varphi}_k}) + \int_{\Omega\cap S^{\sigma}_{\widehat\varphi_k}} |\nu_{\widehat\varphi_k}\cdot\xi|{\rm d}\Hausdorff\right)\\
		& \le c|Du|(\Omega)+ \sigma\Hausdorff(\partial^*F)
		+c\liminf_{k\rightarrow +\infty}\Hausdorff( S^{\sigma}_{\widehat\varphi_k})<+\infty\,,
	\end{split}
\end{equation*}
where the last inequality follows from the fact that $ (S_{\widehat\varphi_k}\setminus S^{\sigma}_{\widehat\varphi_k})\subset S_u\cup \partial^*F$, 
and since $\sigma\leq \pi$, one has   $|\jump{\widehat \varphi(x)}|\leq c|\jump{u(x)}|$ for a.e. $x\in S_{\widehat\varphi_k}\setminus S^{\sigma}_{\widehat\varphi_k}$ and some $c>0$.
Thus for
$\Hausdorff$-a.e. $\xi \in \mathbb S^{n-1}$ and 
 $\Hausdorff$-a.e. $y \in \Pi^\xi$ 
we can consider a 
not relabelled subsequence (depending on $\xi$  and $y$) such that
\begin{equation*}
\sup_k	\left[|D(\widehat\varphi_k)_y^\xi|(\Omega_y^\xi\setminus S^{\sigma}_{(\widehat\varphi_k)_y^\xi} )+\mathcal H^0(\Omega_y^\xi\cap S^{\sigma}_{(\widehat\varphi_k)_y^\xi} )
\right]<+\infty\,.
\end{equation*}
Hence
for not relabelled subsequence (depending on $\xi$ and $y$) 
there exists $m \in \mathbb N$ depending on $\xi$ and $y$ 
such that $\mathcal H^0\left(\Omega_y^\xi\cap S^{\sigma}_{(\widehat\varphi_k)_y^\xi} \right)=m$
for all $k\geq1$, and the jump points converge 
as well in $\overline{\Omega_y^\xi}$ as $k \to +\infty$ for $\Hausdorff$-a.e. 
$\xi \in \mathbb S^{n-1}$ and 
$\Hausdorff$-a.e. $y \in \Pi^\xi$ .
Thus, using the fact that $\widehat \varphi_k\rightarrow +\infty$ in $\Om\setminus E$, and therefore on $\Om_y^\xi\setminus E_y^\xi$,  it is possible to show that
\begin{equation*}
	\liminf_{k\to+\infty}\mathcal H^0\left(
A_y^\xi\cap S^{\sigma}_{(\widehat\varphi_k)_y^\xi}\right)\ge \mathcal H^0\left(A_y^\xi\cap\partial^*
E_y^\xi
	\right)
\end{equation*}
(see \cite{AT} for details).
By integrating over $\Pi^\xi$ and using once more 
Fatou's lemma we conclude \begin{equation*}
	\begin{split}
	\liminf_{k \to +\infty}\int_{A\cap S^{\sigma}_{\widehat\varphi_k}} |\nu_{\widehat\varphi_k}\cdot\xi|{\rm d}\Hausdorff
&\ge
\int_{\Pi^\xi}
\liminf_{k\to+\infty}\mathcal H^0\left(A_y^\xi\cap S^{\sigma}_{(\widehat\varphi_k)_y^\xi}\right){\rm d}\Hausdorff(y)
\\
&\ge 
\int_{\Pi^\xi}\mathcal H^0\left(\partial^*
E_y^\xi	\right){\rm d}\Hausdorff(y) \,,
	\end{split}
\end{equation*}
for $\Hausdorff$-a.e. $\xi\in\mathbb{S}^{n-1}$ and every open 
set $A\subseteq\Omega$. As a consequence
 we deduce that  $E$ has finite perimeter in $\Om$, since 
 by integrating over $\mathbb{S}^{n-1}$,
\begin{equation*}
\alpha_n\liminf_{k \to +\infty}\Hausdorff(A\cap S^{\sigma}_{\widehat\varphi_k})\ge \alpha_n\Hausdorff(A\cap\partial^*E)\,,
\end{equation*}
where $\alpha_n:=\int_{\mathbb S^{n-1}}|\nu\cdot\xi|{\rm d}\Hausdorff(\nu)$, which in turn implies 
\begin{equation}\label{liminf}
	\liminf_{k \to +\infty}\Hausdorff(A\cap S_{\widehat\varphi_k}^\sigma)\ge \Hausdorff(A\cap\partial^*E)\,.
\end{equation}
To prove that $\varphi_\infty\in GBV(\Omega)$ 
we argue as follows. For $N\in\mathbb N$ we set $\widehat{\varphi}_k^N:=(\widehat{\varphi}_k\wedge N)\vee(-N)$, so that 
\begin{equation*}
|	D\widehat{\varphi}_k^N|(\Omega)\le |Du|(\Omega)+ 2N\Hausdorff(S_{\widehat\varphi_k^N})\le c\,.
\end{equation*}
Hence $(\widehat{\varphi}_k^N)$ converges weakly* 
in $BV(\Omega)$ to $\widehat\varphi_\infty^N$ as $k \to +\infty$,
where 
\begin{equation*}
\widehat\varphi_\infty^N
:=\begin{cases}
(\varphi_\infty\wedge N)\vee (-N)&\text{ in }E,
\\
\pm N& \text{ in }\Omega\setminus E,
\end{cases} \quad {\rm so ~ that} \quad 
\widehat\varphi_\infty^N
\in BV(\Omega)\,.
\end{equation*}
Hence observing that in the whole of
$\Omega$ we have $(\varphi_\infty\wedge N)\vee (-N)= \widehat\varphi_\infty^N\chi_E\in BV(\Omega)$ we conclude $\varphi_\infty\in GBV(\Omega)$. Eventually if in addition \eqref{eq:if_furthermore} holds then \eqref{eq:lem5.9:bis}  can be deduced in the 
same way as \eqref{liminf} with $\varphi_k-2\pi d_k$ in place of $\widehat\varphi_k$.
This additionally implies 
$$
P(E,F)\leq \constcompu,
$$
which shows  \eqref{eq:lem5.9:tris}.
\end{proof}

		\begin{lemma}\label{lem:ricoprimenti}
	{Let $0\le\sigma\le\pi$.}		Let  $(\varphi_k)_{k\ge1}\subset 
GBV(\Om)$  be a sequence of functions satisfying 
\begin{equation}\label{eq:uf}
H:= \sup_{k \in \mathbb N} 
\Hausdorff(S^\sigma_{\varphi_k}) < +\infty.
\end{equation}
 Let $N\ge1$ be an integer, $E_1,\dots,E_N\subset \Om$ be 
pairwise disjoint nonempty finite perimeter sets with 
the following property: for any $i=1,\dots,N$, 
			\begin{equation*}\label{eq:52}
				\begin{split}
{\liminf_{k\rightarrow +\infty}
\Hausdorff(S^\sigma_{\varphi_k}\cap A)
\ge \Hausdorff(
A \cap 
\partial^*E_i)
  \text{ for any  } A\in\mathcal A(\Omega)}\,.
				\end{split}
			\end{equation*}
			Then 
\begin{equation*}\label{70thesis}
\constcompu\ge \liminf_{k\rightarrow +\infty}\Hausdorff(S^\sigma_{\varphi_k})\ge  
\Hausdorff\big(\Omega\cap (\cup_{i=1}^N\partial^*E_i )\big)\,.
\end{equation*}

		\end{lemma}
		
	\begin{proof} 
	We define  the superadditive set 
function $\mu\colon\mathcal A(\Omega)\to [0,+\infty)$ as
	\begin{equation*}
		\mu(A):= \liminf_{k\to+\infty}\Hausdorff(A\cap S^\sigma_{\varphi_k})\,,
	\end{equation*}
	and	the positive measure
		\begin{equation*}
		\lambda(\Borelset):=\Hausdorff\big(\Borelset\cap 
(\cup_{i=1}^N\partial^*E_i )\big)\quad \forall  \text{ Borel set }\Borelset
\subseteq\Omega\,.
		\end{equation*}
		Moreover for $i=1,\dots,N$ we set
		\begin{equation*}
		h_i:=\begin{cases}
			1& \text{ in }\partial^*E_i\\[1em]
			0&\text{ otherwise in }\Omega
		\end{cases}\,.
		\end{equation*}
		Clearly we have 
		\begin{equation*}
\mu(A)\ge \int_A h_i\, {\rm d}\lambda\quad \forall i=1,\dots,N\,.
		\end{equation*}
	Therefore, since 
	\begin{equation*}
	h:=\sup_{i=1,\dots, N}h_i=\begin{cases}
		1&\text{ in }\cup_{i= 1}^N \partial^*E_i\\[1em]
0&\text{ otherwise in }\Omega
	\end{cases}\,,
	\end{equation*}
	by Lemma \ref{lem:localisation-lemma} it follows
	\begin{equation*}
\liminf_{k\to+\infty}\Hausdorff(S^\sigma_{\varphi_k}\cap A)=	\mu(A)\ge  \int_Ah\,{\rm d}\lambda=\Hausdorff(A\cap (\cup_{i= 1}^N\partial^*E_i))\,,
	\end{equation*}
which implies the thesis by taking $A=\Omega$.
\end{proof}

\subsection{Proof of Theorem \ref{thm:compactness}}
		
We adapt the arguments of
\cite[Theorem 3.1]{BMS25}.
\\

\noindent \textbf{Base step $M=1$.}
We set 
\begin{equation*}\label{eq:F_1}
F_1:=\Omega.
\end{equation*}
 By Lemma \ref{lem:passo_base} we 
find, after extracting  a not-relabelled subsequence, 
a finite perimeter set $E_1\subseteq F_1$, 
a sequence $(d^{(1)}_k)_{k\ge1}\subset\mathbb Z$  and a function 
$\varphi^{(1)}_\infty\in GBV(\Om)$,
 such that 
\begin{equation*}
\varphi^{(1)}_\infty \text{ 
is a lifting of } u \text{ in } E_1,
\qquad	\varphi^{(1)}_\infty=0\quad\text{ in }F_1\setminus E_1\,,
\end{equation*}
and
\begin{equation*}
		|E_1|
\geq 
\frac{\dimension^\dimension \omega_\dimension
|F_1|^\dimension}{2^n(2
\constcompu
+\Hausdorff(\partial\Om))^\dimension}\,.
\end{equation*} 
Moreover
\begin{equation*}
\begin{aligned}
& (\varphi_k(x)-2\pi d^{(1)}_k)\rightarrow \varphi^{(1)}_\infty(x)
\qquad \text{ for a.e. }x\in E_1\,,\\[1em]
&  	|\varphi_k(x)-2\pi d^{(1)}_k|\rightarrow +\infty\qquad 
\text{ for a.e. }x\in F_1\setminus E_1\,.
\end{aligned}
\end{equation*}
Since $F_1 = \Om$, estimate \eqref{eq:lem5.9:bis} yields
\begin{equation*}\label{eq:final_part}
\liminf_{k\rightarrow +\infty}\Hausdorff
(S^\sigma_{\varphi_k}\cap A)\geq \Hausdorff(
A \cap \partial^*E_1)
\text{ for any open set } A\subseteq F_1\,.
		\end{equation*}
\smallskip
		
\noindent \textbf{Inductive step $M\leadsto M+1$.}  
		Let $M\ge2$. Suppose we have pairwise
		disjoint nonempty finite perimeter sets,
$E_1,\dots,E_{M}\subset\Omega$, 
and define inductively 
$$F_1:=\Omega\,,\quad
F_i:=\Omega\setminus \bigcup_{j=1}^{i-1}E_j \ \ \text{ for } i=2,\dots,M,
$$
so that $E_i\subset F_i$ for all $i=1,\dots,M$.
Assume that, still for all $i=1,\dots,M$, the following holds:
\begin{itemize}
\item[(i)] 
There exists a function $\varphi_\infty^{{(i)}}\in GBV(\Omega)$ 
which is 
a lifting of $u$ in  $E_i$, and 
$\varphi_\infty^{{(i)}}=0$  in $\Om\setminus E_i$;
\item[(ii)] The set $E_i$ satisfies
\begin{equation*}
|E_i|\geq \frac{
\dimension^\dimension \omega_\dimension
|F_i|^\dimension}{2^n(2\constcompu
+\Hausdorff(\partial^*F_i))^\dimension}\geq\frac{
\dimension^\dimension \omega_\dimension
|F_i|^\dimension}{2^n(3\constcompu
+\Hausdorff(\partial\Om))^\dimension};
\end{equation*} 
\item[(iii)]  
There exists a sequence 
$(d_k^{(i)})_{k}\subset \mathbb{Z}$, 
such that 
\begin{equation}
\begin{split}\label{eq:conv-m-insiemi}
&(\varphi_k(x)-2\pi \dki
)\rightarrow \varphi^{(i)}_\infty(x)\qquad \text{ for a.e. }x\in E_i\,,\\[1em]
					&|\varphi_k(x)-2\pi 
\dki
|\rightarrow +\infty\qquad \text{ for a.e. }x\in \Om\setminus E_i\,,\\[1em]
					& \liminf_{k\rightarrow +\infty}
\Hausdorff(S^\sigma_{\varphi_k}\cap A)
\geq 
\Hausdorff(A \cap \partial^*E_i)
\quad\text{ for any } A\in\mathcal A(\Omega)\,.
				\end{split}
\end{equation}
\end{itemize}
		
Set, for $M\ge1$,
$$
F_{M+1}:=\Om\setminus
\Big(\bigcup_{i=1}^{M}E_i\Big), 
$$
If $F_{M+1}=\emptyset$  nothing remains to prove. 
Therefore, we may suppose $F_{M+1}\not=\emptyset$.
 From Lemma \ref{lem:ricoprimenti} 
			\begin{equation}\label{eq:liminf-n-insiemi}
\liminf_{k\rightarrow +\infty}\Hausdorff
(S^\sigma_{\varphi_k})\geq 
\Hausdorff ( \Omega\cap(\cup_{i=1}^M\partial^*E_i))\,,
			\end{equation}
			and thus, since $\partial^*(\cup_{i=1}^ME_i)\subseteq \cup_{i=1}^M\partial^*E_i$,
			\begin{equation}\label{eq:controllo-perimetro-Fm+1}
\Hausdorff\Big(\Omega\cap\partial^*(\cup_{i=1}^ME_i)\Big)= 
\Hausdorff(\Omega \cap\partial^* F_{M+1})\le \constcompu\,.
			\end{equation}
		
Applying
 {Lemma \ref{lem:passo_base}} to the set $F=F_{M+1}$, we obtain 
 a 
finite perimeter set $$E_{M+1}\subseteq F_{M+1}\,,$$ a function 
$\varphi_\infty^{M+1}\in GBV(\Omega)$  and a sequence $(d_k^{M+1})_{k\ge1} \subset \mathbb Z$, such that  \begin{equation*}
			\varphi^{M+1}_\infty=0\quad\text{ in }\Omega\setminus E_{M+1}\,,\quad \varphi_\infty^{M+1}\text{ is a lifting of } u \text{ in } E_{M+1}\,,
		\end{equation*}
	and
		\begin{equation}\label{eq:controllo-area-Fm+1}
			|E_{M+1}|\geq \frac{\dimension^\dimension
\omega_\dimension|F_{M+1}|^n}{2^n(2\constcompu
+\Hausdorff(\partial^*F_{M+1}))^n}\,,
		\end{equation}
together with
		\begin{equation*}\label{eq:lem5.9BIS}
			\begin{split}
				&(\varphi_k(x)-2\pi d_k^{M+1})\rightarrow \varphi^{M+1}_\infty(x)\qquad \text{ for a.e. }x\in E_{M+1}\,,\\[1em]
				&  	|\varphi_k(x)-2\pi d^{M+1}_k|\rightarrow +\infty\qquad \text{ for a.e. }x\in F_{M+1}\setminus E_{M+1}\,,\\[1em]
				&\liminf_{k\rightarrow +\infty}
\mathcal H^{n-1}(S^\sigma_{\varphi_k}\cap B)
\geq \mathcal H^{n-1}(
B \cap \partial^*E_{M+1})
\text{ for any open set } B
\subseteq F_{M+1}\,.
			\end{split}
		\end{equation*}
{Gathering \eqref{eq:controllo-perimetro-Fm+1} and \eqref{eq:controllo-area-Fm+1} we have
			\begin{equation*}
				|E_{M+1}|\geq \frac{\dimension^\dimension
\omega_\dimension|F_{M+1}|^n}{2^n(2\constcompu
+\Hausdorff(\partial^*F_{M+1}))^n}\geq\frac{\dimension^\dimension \omega_\dimension|F_{M+1}|^n}{2^n(3\constcompu
+\Hausdorff(\partial\Om))^n} >0\,.
			\end{equation*}
		}
		Moreover by \eqref{eq:conv-m-insiemi}, also
		\begin{align*}
			|\varphi_k(x)-2\pi d^{M+1}_k|\rightarrow +\infty\qquad \text{ for a.e. }x\in \Om\setminus E_{M+1}\,.
		\end{align*}
		Thus 
properties (i)-(iii) are preserved at level $M+1$.
\\
		
		\noindent \textbf{Conclusion.} 
	Iterating the above 
construction and extracting a diagonal subsequence, we obtain 
 a sequence $(E_i)_{i\ge1}$ of mutually disjoint finite perimeter sets in $\Omega$ such that for every $M\ge1$, $E_1,\dots,E_M$ satisfy 
properties (i)-(iii). 
From \eqref{eq:liminf-n-insiemi},
\begin{equation}\label{liminf-passo-n}
				\liminf_{k\rightarrow +\infty}
\Hausdorff(S^\sigma_{\varphi_k})\geq \Hausdorff( 
\Omega\cap(\cup_{i=1}^
\indiceteo\partial^* E_i)) \qquad\forall \indiceteo
\ge1\,.
			\end{equation}
		To conclude, we  show that $$|\Om\setminus (\cup_{i=1}^{+\infty} E_i)|=0.$$ Since $\sum_{\indiceteo=1}^{+\infty}
|E_\indiceteo|<+\infty$, the sequence 
$(|E_\indiceteo|)$ tends to zero as $\indiceteo
\rightarrow +\infty$,
and using
$$|E_\indiceteo|\geq \frac{\dimension^\dimension \omega_\dimension
|F_\indiceteo|^2}{2^\dimension(3\constcompu
+\Hausdorff(\partial \Om))^2}$$  we deduce that $|F_\indiceteo|\rightarrow 0$. 
In particular, $$|\Om\setminus (\cup_{i=1}^{+\infty} E_i)|=\lim_{\indiceteo
\rightarrow+\infty}|\Om\setminus (\cup_{i=1}^{\indiceteo
-1} E_i)|=\lim_{\indiceteo
\rightarrow +\infty}|F_\indiceteo|=0\,.$$ Define
		$$\varphi_\infty(x):=\varphi_\infty^{\pip}(x)\qquad \text{ 
if }x\in E_i {\rm ~for ~some~} i \in \NN.$$
Then $\varphi_\infty$ is a lifting of $u$ in $\Omega$. 
We now show that 
\begin{equation}\label{eq:phiGBV}
\varphi_\infty\in GBV(\Om).
\end{equation}
First observe that by letting $m\to\infty$ in \eqref{liminf-passo-n} 
we deduce that 
		\begin{equation}\label{liminf-perimeter}
H\ge	\liminf_{k\rightarrow +\infty}
\Hausdorff(S^\sigma_{\varphi_k})\geq \Hausdorff( 
\Omega\cap(\cup_{i=1}^
\infty\partial^* E_i))\,.
		\end{equation}
		Now, for each $N\in \mathbb N$ we define the auxiliary sequence 
		\begin{equation*}
\Phi_k^N:=
((\varphi_k-2\pi d_k^{(i)})\wedge N)\vee(-N) \quad\text{ in}\quad E_i \quad\text{for all }i\ge1\,.
		\end{equation*}
		Thus, noticing that $S_{\Phi_k^N}\subset S_u\cup S^\sigma_{\varphi_k}\cup (\cup_{i=1}^\infty\partial^*E_i)$
		 we have from \eqref{liminf-perimeter}
		\begin{equation*}
|D\Phi_k^N|(\Omega)\le \int_\Omega|\nabla u|\dx+ |D^cu|(\Omega)+\int_{S_u}|u^+-u^-|{\rm d}\Hausdorff+ 2N\Hausdorff(S^\sigma_{\varphi_k})
\le C\,,
		\end{equation*}
with $C>0$ depending on $N$.
Therefore, up to a subsequence, $(\Phi_k^N)$ converges  to
$(\varphi_\infty\wedge N)\vee (-N)$ weakly* in $BV(\Omega)$ as $k
\to +\infty$. By the 
arbitrariness of $N$ we deduce \eqref{eq:phiGBV}.

\qed

\subsection{Proof of Theorem \ref{teo:compactness}}
\label{sec:compactness_of_liftings_of_a_converging_sequence}
	Let $\overline \varphi,\overline \varphi_k\in BV(\Om)\cap L^\infty(\Om)$ be liftings of $u$ and of $u_k$ in $\Om$, respectively with 
\begin{equation*}\label{eq:varphi_BV}
|\overline \varphi|_{BV}\leq 2|u|_{BV}\leq C,\qquad\qquad |\overline \varphi_k|_{BV}\leq 2|u_k|_{BV}\leq C
\end{equation*}
(see Theorem \ref{thm:dav-ign}). 
The constant $C>0$ can be found
using the assumption 
\eqref{eq:dododo}
and so $u_k$ are uniformly bounded in $BV(\Omega)$.  
In particular, for  $\sigma\in[0,\pi]$ as in 
the statement, there is a positive  constant $\newconst$ such that 
\begin{align*}
	\sup_{k\geq1} \{\mathcal H^{n-1}(S^\sigma_{\overline \varphi_k})
+\mathcal H^{n-1}(S^\sigma_{\overline \varphi})\}\leq \newconst.
\end{align*} 
Furthermore, $\overline \varphi_k$ have equibounded $BV$-norm, and so there is a (not relabelled) subsequence such that 
\begin{align}\label{convoverline}
\overline \varphi_k\rightarrow \overline \varphi_\infty\text{ weakly$^*$ in }BV(\Om); 
\end{align}
since up to a subsequence the same convergence holds pointwise a.e.,
and also $u_k\rightarrow u$ a.e. on $\Om$, we deduce that $\overline\varphi_\infty$ must be a lifting of $u$.

Thanks to the fact that $\sigma\leq\pi$,  the functions $$v_k:=\varphi_k-\overline \varphi_k,\qquad k\geq1$$ take values in $2\pi\mathbb Z$, and belong to $GSBV(\Om;2\pi\mathbb Z)$, since the jump sets satisfy
$S_{v_k}\subseteq S^\sigma_{\overline \varphi_k}\cup S^\sigma_{\varphi_k},$
and so 
$$\mathcal H^{n-1}(S_{v_k})\leq \mathcal H^{n-1}(S^\sigma_{\overline \varphi_k})+\mathcal H^{n-1}( S^\sigma_{\varphi_k})<
\newconst <+\infty,$$ for all $k\geq1$.
Trivially $v_k$ are liftings of the same 
function $f\equiv(1,0)\in \mathbb S^1$, and so we can apply 
Theorem \ref{thm:compactness} to obtain that, for a non-relabelled subsequence, there exist a Caccioppoli 
partition $\{E_i\}_{i\in \mathbb N}$ of $\Om$, 
sequences of integers $(d^{(i)}_k)_{k\geq1}$, and 
a function $v_\infty\in GSBV(\Om;2\pi\mathbb Z)$ such that, for all $i \in \mathbb N$,
\begin{align*}
&(\varphi_k-\overline \varphi_k-2\pi d_k^{(i)})\rightarrow v_\infty \text{ pointwise a.e. in }E_i, \\
&|\varphi_k-\overline \varphi_k-2\pi d_k^{(i)}|\rightarrow +\infty \text{ pointwise a.e. in }\Om\setminus E_i.
\end{align*}
From this and \eqref{convoverline} we conclude
\begin{align*}
	&(\varphi_k-2\pi d_k^{(i)})\rightarrow v_\infty+\overline \varphi_\infty \text{ pointwise a.e. in }E_i, \\
	&|\varphi_k-2\pi d_k^{(i)}|\rightarrow +\infty \text{ pointwise a.e. in }\Om\setminus E_i.
\end{align*}
This  is the thesis, just by setting $\varphi_\infty:=\overline \varphi_\infty+v_\infty$, which is a lifting of $u$ 
since $v_\infty$ takes values in $2\pi\mathbb Z$.
\qed

\section{$\Gamma$-convergence of functionals 
on $\mathbb{S}^1$- valued maps}
\label{sec:Gamma_convergence}
The proof of Theorem \ref{thm2} is delicate. In particular
 the lower bound inequality requires a local argument which relies
 on the compactness result 
for liftings (Theorem \ref{teo:compactness}).
For convenience we introduce the localised Modica-Mortola-type (or Allen-Cahn type) functionals
\begin{equation*}
\ModicaMortolaeps
(v, \sottoaperto)
:=\int_{\sottoaperto}\left(\var |\nabla v|^2+\frac{W(v)}{\var}
\right)\dx \qquad \forall v \in W^{1,2}(\Om),
\end{equation*}
for every open set $\sottoaperto\subseteq\Omega$, where $W$ 
is defined in \ref{hyp:W} of
Section \ref{sec:Gamma_approximation_and_compactness_for_the_MS_functional}.

\subsection{Proof of Theorem \ref{thm2}}

\noindent 
\textit{Step 1: Lower bound}. 
Let $\var_k\searrow0$ as $k\to+\infty$.
We have to show that
\begin{equation}\label{lw_bound2}
	\liminfk
\ATtwoepsk(u_k,v_k)\ge 
\MSlift(u,v)
\end{equation}
for every sequence $((u_k,v_k))_{k\ge1}\subset L^1(\Omega;\mathbb{S}^1)\times L^1(\Omega)$ 
converging to $(u,v)$ in $L^1(\Omega;\mathbb{S}^1)\times L^1(\Omega)$.
We may assume 
\begin{equation*}\label{eq:assume_C}
\sup_{k \in \NN} 
\ATtwoepsk(u_k,v_k)\leq C <+\infty
\end{equation*}
 so that $(u_k,v_k)\in\domATpiccolo$, $v=1$ a.e. in $\Omega$
and, 
up to a not relabelled subsequence, 
\begin{equation*}
	\liminfk
\ATtwoepsk
(u_k,v_k)=	\limk
\ATtwoepsk
(u_k,v_k).
\end{equation*}
From Theorem~\ref{thm:Linear-Ambrosio-Tortorelli} 
it follows that 
\begin{equation*}
	\begin{split}
	\liminf_{k\to+\infty}
\ATtwoepsk
(u_k,v_k)\ge \int_\Omega f( |\nabla u| )\dx+|D^cu|(\Omega)+\int_{S_u}g(|u^+-u^-|){\rm d}\Hausdorff\,.
	\end{split}
\end{equation*}
In particular, from Remark \ref{rem:decoupling}-(ii)-$5$
 and the fact that $|u^+-u^-|\le 2$ a.e. 
(since $u$ is $\mathbb{S}^1$-valued) we deduce $u\in BV(\Omega;\mathbb{S}^1)$.

For every $k\ge1$, since $\Omega$ simply connected, we can choose a lifting
$\varphi_k\in W^{1,2}(\Omega)$ of $u_k$ in $\Om$.
Using  that $|\nabla u_k|=|\nabla \varphi_k|$ we have
\begin{equation*}
	+\infty>C\ge \ATtwoepsk(u_k,v_k)=\int_\Omega\left(
\psi(	v_k) f(|\nabla \varphi_k|)+ \var_k|\nabla v_k|^2+\frac{W(v_k)}{\var_k}\right)\dx\,.
\end{equation*}
Thus, from the coarea formula,
for any $k \geq 1$, 
\begin{equation}\label{eq:etaprimo}
C\geq \MMepsk(v_k, \Om)
\ge \int_\Omega \sqrt{W(v_k)}|\nabla v_k|\dx \\
= \int_0^1\sqrt{W(t)}\Hausdorff(\partial^* F_k^t)\dt
\end{equation}
where $F_k^t:=\{x\in \Om:v_k(x)\leq t\}$. Now,
fix $\etaprimo,\etasecondo
\in(0,1)$ with $\etaprimo<\etasecondo$.
By \eqref{eq:etaprimo} and the mean value theorem 
there exist $\livellok
\in (\etaprimo,
\etasecondo)$ and $c(\etaprimo, \etasecondo)>0$ independent of $k$  
such that 
\begin{equation}\label{eq:etaeta}
C\ge c(\etaprimo,\etasecondo)
(\etasecondo-\etaprimo)
\Hausdorff(\partial^*F_k^{\livellok})\,.
\end{equation}
Moreover, using also that $\psi$ is increasing,
\begin{equation}\label{eq:due}
C\ge \int_\Omega \psi(v_k)f( |\nabla\varphi_k|)\dx\ge C
\psi(\etaprimo)\int_{\Omega}
\chi_{\Omega\setminus F^{\livellok}_k}
|\nabla\varphi_k|
\dx\,.
\end{equation}
Setting 
$$
\nuovoliftingk
:=\varphi_k\chi_{\Omega\setminus F^{\livellok}_k}\in 
SBV^2(\Omega),
$$
 we have 
$$
S_{\nuovoliftingk}\subseteq \partial^*F_k^{\livellok} \quad {\rm  and }\quad
\nabla\nuovoliftingk
=\nabla \varphi_k\chi_{\Omega\setminus F_k^{\livellok}}.
$$
Therefore,
from \eqref{eq:due} and \eqref{eq:etaeta}, 
\begin{equation}\label{eq:pallino}
	\int_\Omega|\nabla \nuovoliftingk
|\dx+ \Hausdorff(S_{\nuovoliftingk
})\le C(\etaprimo, \etasecondo)
\end{equation}
for some $C(\etaprimo, \etasecondo)>0$ 
depending on $\etaprimo,\etasecondo$ and independent of $k$. Define
$$
\overline 
u_k:=e^{i\nuovoliftingk
}=u_k\chi_{\Omega\setminus F_k^{\livellok}}+(1,0)\chi_{F_k^{\livellok}} \qquad
\forall k \in \mathbb N,
$$
which, by 
\eqref{eq:etaeta} and \eqref{eq:due}, are 
uniformly bounded in $BV(\Om; \Suno)$. This
and
\begin{equation}\label{eq:zero_in_measure}
\vert F_k^{\livellok}\vert \to 0,
\end{equation} which is a consequence of $t(k)<\eta''<1$ and $v_k\to1$ a.e.,
imply that the sequence $(\overline u_k)$ weakly$^{\star}$ converges to $u$ 
in $BV(\Omega;\mathbb{S}^1)$. Hence, using \eqref{eq:pallino},
we can apply Theorem \ref{teo:compactness} 
to the sequence
$(\nuovoliftingk
)_{k\ge1}$ and get, for a not-relabelled subsequence, a  
lifting $\varphi_\infty\in GBV(\Omega)$ 
of $u$ in $\Omega$, such that $(\phi_k)$ converges 
locally modulo $2\pi$ to $\varphi_\infty$.
Namely, there exists a Caccioppoli
partition $\{E_i\}_{i\in \mathbb{N}}$ of $\Omega$, sequences $(d^{\pip}_k)_{k\ge1}\subset\mathbb{Z}$ 
for any integer $i\ge1$ with the following properties:
	\begin{equation*}
	\begin{split}
		&\lim_{k \to +\infty}(\nuovoliftingk
(x)-2\pi d^{\pip}_k)= \varphi_\infty(x)\qquad 
\forall i \in \NN, 
\text{ for a.e. }x\in E_{i}\,,\\
		&\lim_{k \to +\infty}|\nuovoliftingk
(x)-2\pi d^{\pip}_k|=+\infty\qquad \forall i \in \NN, 
\text{ for a.e. }x\in \Om\setminus E_i.
	\end{split}
\end{equation*}
Again, using
\eqref{eq:zero_in_measure},
the same holds for $\varphi_k$, i.e.,
	\begin{equation}\label{eq:conv-in-Ei}
	\begin{split}
		&\lim_{k \to +\infty}(\varphi_k(x)-2\pi d^{\pip}_k)= \varphi_\infty(x)
\qquad 
\forall i \in \NN, 
\text{ for a.e. }x\in E_{i}\,,\\
		&\lim_{k \to +\infty}|\varphi_k(x)-2\pi d^{\pip}_k|= +\infty\qquad 
\forall i \in \NN, 
\text{ for a.e. }x\in \Om\setminus E_{i}.
	\end{split}
\end{equation}
In particular, to prove the validity of \eqref{lw_bound2} it is sufficient to show  
 \begin{equation*}\label{liminf-MM-bis}
 	\liminf_{k\to+\infty}
 \ATtwoepsk
 (u_k,v_k)\ge \int_\Omega f( |\nabla \varphi_\infty| )\dx+|D^c\varphi_\infty|(\Omega)+\int_{S_u}g(|\varphi_\infty^+-\varphi_\infty^-|){\rm d}\Hausdorff \,.
 \end{equation*}
since, being $\varphi_\infty$ a lifting of $u$, we have 
$$ \int_\Omega f( |\nabla \varphi_\infty| )\dx
=\int_\Omega f( |\nabla u| )\dx, \qquad
|D^c\varphi_\infty|(\Omega)= 
|D^cu|(\Omega)\,,$$
and 
$$
\int_{S_u}g(|\varphi_\infty^+-\varphi_\infty^-|){\rm d}\Hausdorff \geq m_g[u]\,.
$$
 For fixed $i$ and $k$ and any integer $K\ge1$ we consider the truncated function 
\begin{equation*}\label{eq:tru}
 	\trunkphi_k^{i,K}:=((\varphi_k-2\pi d_k^{\pip})\wedge K)\vee (-K)\in W^{1,2}(\Omega)\,,
 \end{equation*}
and 
$$
\phi_k^{i,K}:= \trunkphi_k^{i,K}\chi_{\Omega\setminus F_k^{\livellok}}.
$$

According to the inclusion 
$S_{\phi_k^{i,K}}\subseteq\partial^*F_k^{\livellok}$ and the inequality
$|\jump{\phi_k^{i,K}}|\le 2K$, we have
\begin{equation*}\begin{split}
|D\phi_k^{i,K}|(\Omega)&= \int_\Omega|\nabla \nuovoliftingk
|\chi_{\{|\varphi_k-2\pi d_k^{\pip}|<K\}}\dx+ \int_{S_{\phi_k^{i,K}}}\jump{\phi_k^{i,K}}\,{\rm d}\mathcal{H}^{n-1} \\
&\le \int_\Omega|\nabla \nuovoliftingk
|\dx+ K\mathcal{H}^{n-1}(\partial^* F_k^{\livellok})\le C\,.
	\end{split}
\end{equation*}
Hence, up to a subsequence (depending on $K$ {and on $i$}), $(\phi_k^{i,K})$ 
converges 
to some $\phi^{i,K}_\infty$ in $L^1(\Omega)$  as $k \to +\infty$. 
Moreover 
from \eqref{eq:conv-in-Ei} it holds 
$$
\phi^{i,K}_\infty
:= (\varphi_\infty\wedge K)\vee (-K) \quad{\rm in~} E_i. 
$$
Let $F_i^\pm$ be 
such that
$$
\Omega\setminus E_i=F_i^+\cup F_i^- \quad {\rm and}\quad
  \phi^{i,K}_\infty=\pm K {\rm ~in~} F_i^{\pm}.
$$
{}From \eqref{eq:zero_in_measure},
it follows that $(\varphi_k^{i,K})$ 
converges to $\phi_\infty^{i,K}$ in $L^1(\Omega)$
 as $k \to +\infty$.
Hence the sequence
 $((\varphi_k^{i,K},v_k))$ converges
to $(\phi_\infty^{i,K},1)$ in $L^1(\Omega)\times L^1(\Omega)$ and 
\begin{equation*}
	\begin{split}
\int_\sottoaperto \left(\psi(v_k) f(|\nabla \varphi_k^{i,K}| )+ \var_k|\nabla v_k|^2+\frac{W(v_k)}{\var_k}\right)\dx\le C\,,
	\end{split}
\end{equation*}
for any open set $\sottoaperto\subseteq\Omega$, for some 
$C=C(A)>0$ independent of $k$. 
From  Theorem \ref{thm:Linear-Ambrosio-Tortorelli} and Remark \ref{rem:decoupling}-(i) for every $i\in \mathbb N$
\begin{equation}\label{conv-AT-lin}
	\begin{split}
			\liminf_{k \to +\infty}& \int_A\left(\psi(v_k) f(|\nabla \varphi_k| )+ \var_k|\nabla v_k|^2+\frac{W(v_k)}{\var_k}\right)\dx\\
		&\ge
	\liminf_{k \to +\infty} \int_A\left( \psi(v_k) f(|\nabla \varphi_k^{i,K}| )+ \var_k|\nabla v_k|^2+\frac{W(v_k)}{\var_k}\right)\dx\\
	&\ge 
	\int_A f( |\nabla \phi_\infty^{i,K}| )\dx+|D^c\phi_\infty^{i,K}|(A)+\int_{S_{\phi_\infty^{i,K}}\cap A}
	g(|(\phi_\infty^{i,K})^+-(\phi_\infty^{i,K})^-|){\rm d}\Hausdorff \,.
	\end{split}
\end{equation}
In particular Remark \ref{rem:decoupling} together with $|(\phi_\infty^{i,K})^+-(\phi_\infty^{i,K})^-|\le2K$ imply 
$$
\phi_\infty^{i,K}\in BV(\Omega).
$$
Consider the 
bounded positive measure  given by
\begin{equation*}
	\lambda(\Borelset):=\int_\Borelset
f (|\nabla \varphi_\infty|) \dx+ |D^c\varphi_\infty|(\Borelset)
+ \mathcal{H}^{n-1}(S_{\varphi_\infty}\cap \Borelset)\quad \forall 
\text{ Borel set }\Borelset\subseteq\Omega\,.
\end{equation*}
{
For each $i, K\in \mathbb N$ we define
\begin{equation*}
	h^{i,K}
:=\begin{cases}
		g(|(\phi_\infty^{i,K})^+-(\phi_\infty^{i,K})^-|)	&\text{in } S_{\phi_\infty^{i,K}}\\[1em]
		1	& \text{otherwise in }\Omega.
	\end{cases}\,
\end{equation*}
}
Thus from \eqref{conv-AT-lin} 
we have 
\begin{equation*}
		\liminf_{k \to +\infty} \int_A\left(\psi(v_k )f(|\nabla \varphi_k| )+ \var_k|\nabla v_k|^2+\frac{W(v_k)}{\var_k}\right)\dx\ge \int_Ah^{i,K}{\rm d}\lambda\quad\forall i,K \in \mathbb N,
\end{equation*}
for all $A\in\mathcal{A}(\Omega)$.
Next, as $|(\phi_\infty^{i,K})^+-(\phi_\infty^{i,K})^-|\rightarrow |(\phi_\infty^{i})^+-(\phi_\infty^{i})^-| $ as $K\rightarrow \infty$ pointwise $\mathcal H^1$-a.e. on $S_{\varphi^i_{\infty}}$, and  then  observing that
\begin{equation*}
	\sup_{i,K \in \mathbb N}h^{i,K}=
\begin{cases}
	g(|\varphi_\infty^+-\varphi_\infty^-|)& \text{in }  S_{\varphi_\infty},
\\[1em]
	1&\text{in } \Om\setminus S_{\varphi_\infty},
	\end{cases}
\end{equation*}
we conclude by Lemma \ref{lem:localisation-lemma}.

\black

\medskip

\noindent
\textit{Step 2: Upper bound}. Let  $\var_k\searrow0$ and 
 $u\in BV(\Omega;\Suno)$. We 
have to find a sequence $((u_k,v_k))\subset
\domATpiccolo$ converging to $(u,1)$ in $L^1(\Omega;\Suno)\times L^1(\Omega)$ such that
\begin{equation*}\label{up_bound2}
	\limsupk 
\ATtwoepsk
(u_k,v_k) \le \MSlift
(u,1)\,.
\end{equation*}

By Corollary \ref{cor:existence-minimizer} 
we can select a jump minimizing lifting 
$\varphi \in GBV(\Omega)$ of $u$ in $\Omega$,  
and 
\begin{equation}\label{quasi-min}
\int_{S_\varphi}g(|\varphi^+-\varphi^-|){\rm d}	\Hausdorff=m_g[u]\,.
\end{equation}

By Theorem \ref{thm:Linear-Ambrosio-Tortorelli} there exist 
 $(\varphi_k,v_k)\in W^{1,2}(\Omega) \times W^{1,2}(\Omega)$ 
such that $(\varphi_k,v_k)\to(\varphi,1)$ in $L^1(\Omega) \times L^1(\Omega)$ and 
\begin{equation}\label{at-varphi}
\limsupk {\rm F}_{\var_k}(\varphi_k,v_k)
\le {\rm F}(\varphi,1)\,.
\end{equation}
Next we let $u_k:=e^{i\varphi_k}\in W^{1,2}(\Omega;\Suno)$, so 
that $(u_k,v_k)\to(u,1)$ in $L^1(\Omega;\Suno)
\times L^1(\Omega)$. 
 Moreover, from \eqref{quasi-min}, \eqref{at-varphi}, $|\nabla u_k|=|\nabla\varphi_k|$ and $|\nabla u|=|\nabla\varphi|$
we get $${\rm F}_{\var_k}(\varphi_k,v_k)=\ATtwoepsk
(u_k,v_k)\,,\quad {\rm F}(\varphi,1)={\rm F}_{\rm lift}(u,1)$$ and so
\begin{equation*}
	\limsupk \ATtwoepsk
(u_k,v_k) \le {\rm F}_{\rm lift}(u,1)\,.
\end{equation*}
\qed
{\subsection{An extension of Theorem \ref{thm2}}
	We consider the domain 
$\widehat{\mathcal D}_{\mathbb S^1} \supsetneq \mathcal D_{\mathbb S^1} $, defined as
\begin{equation}\label{eq:extensiondomain}
	\begin{aligned}
	\widehat{\mathcal D}_{\mathbb S^1}
		:=\left\{(u,v)\in W_{\text{{\rm loc}}}^{1,2}(\Omega;\mathbb{S}^1)\times W_{\text{{\rm loc}}}^{1,2}(\Omega)\colon \ 0\le v\le1
		\right\}\subset 
		L^1(\Om; \circle) \times L^1(\Om)\,.
	\end{aligned}
\end{equation}
and the extended energy
\begin{equation}\label{AT3}
\widehat\ATtwoeps(u,v):=
	\begin{cases}
		{\rm F}_\var(u,v)& \text{ if }(u,v)\in 	\widehat{\mathcal D}_{\mathbb S^1}\,,\\[1em]
		+\infty& \text{ otherwise in }L^1(\Om; \circle) \times L^1(\Om)\,.
	\end{cases}
\end{equation}

Then, the following result holds:
\begin{theorem}
[$\Gamma$-convergence of $
\widehat \ATtwoeps$]\label{thm3}
{ Provided $\Omega$ is   
	simply-connected,} we have 
$$
\Gamma-L^1 \lim_{\var\to 0^+} \widehat
\ATtwoeps = \MSlift.
$$
\end{theorem}

The proof of the upper bound in Theorem \ref{thm3} is straightforward 
since 
$\widehat{\mathcal D}_{\mathbb S^1}\supsetneq	{\mathcal D}_{\mathbb S^1}$. To 
obtain the lower bound, it suffices to follow the proof of Theorem \ref{thm2} 
for a sequence $A_k\subset\subset \Om$ with $A_k\nearrow\Om$, and  apply the locality property in Remark \ref{rem:decoupling}(ii). The thesis follows by a diagonal argument.
}
\section{Connections with the relaxed area  in $\mathbb S^1$}\label{sec:relaxed-area}

The relaxation of the
area functional in $\mathbb{S}^1$ for the graph
of an $L^1$-function $u\colon\Omega\to \mathbb S^1$ is defined as
\begin{equation*}
	\mathcal{A}^{\mathbb S^1}
(u,\Omega):=\inf\left\{\liminf_{k \to +\infty}
	\int_\Omega\sqrt{1+|\nabla u_k|^2}\dx\colon u_k\in C^1(\Omega;\mathbb{S}^1)\,,\ u_k\to u\text{ in } L^1
	\right\}\,.
\end{equation*}
In \cite[page 613]{Giaq} it is observed that 
 $\mathcal{A}^{\mathbb S^1}(u,\cdot)$ is not subadditive as a set function, 
therefore it can not 
admit an integral representation and a characterization using 
Cartesian currents is provided in  \cite[page 612, Theorem 1]{Giaq}.
We recall that,
as a consequence of \cite[page 612, Theorem 1]{Giaq},
	$\mathcal{A}^{\mathbb{S}^1}(u,\Omega)$ is finite if and only if
 $u\in BV(\Omega;\mathbb{S}^1)$.

 Here we propose
a characterization 
using lifting theory when $\Omega$ is simply connected and leave the general case for future developments.
The expression we obtain clearly reveals the nonlocal nature
of 
$\mathcal{A}^{\mathbb S^1}(u,\cdot)$.
\begin{theorem}[Characterization of $\mathcal{A}^{\mathbb S^1}$]\label{thm:relaxation-area}
	Let $\Omega\subset\R^n$ be a bounded connected 
simply connected open set with Lipschitz boundary.	
Then
	\begin{equation}
		\mathcal{A}^{\mathbb S^1}(u,\Omega)= \int_\Omega\sqrt{1+|\nabla u|^2}\dx+ |D^cu|(\Omega)+ m[u]\,, \qquad u \in BV(\Om; \mathbb S^1),
	\end{equation}
	where 
	\begin{equation}
		m[u]:=\min\left\{\int_{S_\varphi}|\varphi^+-\varphi^-|{\rm d}\Hausdorff\colon \varphi\in BV(\Omega)\,,\, e^{i\varphi}=u \text{ a.e. in }\Omega
		\right\}\,.
	\end{equation}
\end{theorem}
We preliminarly observe that
	if $u\in W^{1,1}(\Omega;\mathbb{S}^1)\cap C^1(\Omega;\mathbb{S}^1)$, then there exists a lifting $\varphi\in W^{1,1,}(\Omega)\cap C^1(\Omega)$ of $u$ such that $-\pi< \fint_\Omega\varphi\dx\le\pi$.
Indeed, the existence of a lifting $\widehat \varphi\in W^{1,1,}(\Omega)\cap C^1(\Omega)$ follows by well-known results \cite{Brezis-Mironescu}. To achieve the latter property it suffices to add to $\widehat \varphi$ a suitable integer multiple of $2\pi$.

\begin{proof}[Proof of Theorem \ref{thm:relaxation-area}]
	
	\noindent\textit{Step 1: Lower bound.} 
	Let $u_k\in C^1(\Omega;\mathbb{S}^1)$ be such that $u_k\to u$ in $L^1$. 
	We need to prove that 
	\begin{equation}\label{eq:lower-bound}
		\liminf_{k \to +\infty}\int_\Omega\sqrt{1+|\nabla u_k|^2}\dx\ge  \int_\Omega\sqrt{1+|\nabla u|^2}\dx+ |D^cu|(\Omega)+ m[u]\,.
	\end{equation}
	
	Assume without loss of generality that 
	\begin{equation*}
		\sup_k\int_\Omega\sqrt{1+|\nabla u_k|^2}\dx<+\infty\,.
	\end{equation*}
	Since $\int_\Omega\sqrt{1+|\nabla u_k|^2}\dx\ge \int_\Omega|\nabla u_k|\dx$ and $|u_k|=1$
	we have $u_k\in W^{1,1}(\Omega;\mathbb{S}^1)$.
For each $k \in \mathbb N$ select a lifting $\varphi_k\in W^{1,1,}(\Omega)\cap 
C^1(\Omega)$ of $u_k$ in $\Om$ with $-\pi\le\fint_\Omega\varphi_k\dx\le \pi$.
	We have 
	\begin{equation*}
		\int_\Omega|\nabla\varphi_k|\dx= \int_\Omega|\nabla u_k|\dx\le C\,.
	\end{equation*}
	Furthermore, by Poincaré inequality 
	\begin{equation*}
		\int_\Omega|\varphi_k|\dx\le \int_\Omega\left|
		\varphi_k-\fint_\Omega\varphi_k dy
		\right|\dx+ |\Omega|\left|\fint_\Omega\varphi_k dx\right|
\le C\int_{\Omega}|\nabla \varphi_k|dx+ |\Omega|\pi\le C\,.
	\end{equation*}
	Thus, 
up to a not relabelled subsequence, 
$\varphi_k\stackrel{*}{\rightharpoonup}\varphi$ in $BV(\Omega)$,
where $\varphi$ is 
a lifting of $u$ in $\Omega$.
	By \cite{AcDa:94} we conclude 
\begin{equation*}
\begin{aligned}
\liminf_{k \to +\infty}\int_\Omega\sqrt{1+|\nabla u_k|^2}\dx
& = 
\liminf_{k \to +\infty}\int_\Omega\sqrt{1+|\nabla \varphi_k|^2}\dx
\\
& \ge \int_\Omega\sqrt{1+|\nabla \varphi|^2} \dx
+ |D^c\varphi|(\Omega)+ \int_{S_\varphi}|\varphi^+-\varphi^-|
d{\mathcal H}^{n-1}
\\			
&\ge \int_\Omega\sqrt{1+|\nabla u|^2} \dx
+ |D^c u|(\Omega)+m[u]\,.
\end{aligned}
\end{equation*}

	\noindent\textit{Step 2: Upper bound.}  Let $u\in BV(\Omega;\mathbb S^1)$. We need to find $u_k\in C^1(\Omega;\mathbb{S}^1)$ such that $u_k\to u$ in $L^1$ and 
	\begin{equation}\label{eq:upper-bound}
		\limsup_{k \to +\infty}\int_{\Omega}\sqrt{1+|\nabla u_k|^2}\dx\le
		\int_\Omega\sqrt{1+|\nabla u|^2}\dx+ |D^cu|(\Omega)+ m[u]\,. 
	\end{equation}
	
	By \cite[page 25]{Brezis-Mironescu2}
	there exists a lifting 
$\tilde\varphi\in BV(\Omega)$ of $u$ in $\Om$  such that
	\begin{equation*}
		|D\tilde\varphi|(\Omega)=\inf\{|D\varphi|(\Omega)\colon \varphi\in BV(\Omega)\,,\ e^{i\varphi}=u\text{ a.e. in }\Omega\}\,.
	\end{equation*}
	Moreover $|D\tilde\varphi|(\Omega)\le 2|D u|(\Omega)$ and $0\le\fint_\Omega\tilde\varphi\dx\le 2\pi$.
	From the identities
	\begin{equation*}
		\int_\Omega|\nabla 
\varphi|\dx=\int_\Omega|\nabla u|\dx\quad\text{ and }
\quad |D^c \varphi|(\Omega)= |D^cu|(\Omega)
	\end{equation*}
	it follows that 
	\begin{equation*}
		\int_{S_{\tilde\varphi}}|\tilde\varphi^+-\tilde\varphi^-|{\rm d}\Hausdorff\le 
		\int_{S_{\varphi}}|\varphi^+-\varphi^-|{\rm d}\Hausdorff
	\end{equation*}
	for any lifting $\varphi\in BV(\Omega)$ 
of $u$ in $\Om$. Hence 
	\begin{equation*}
		\int_{S_{\tilde\varphi}}|\tilde\varphi^+-\tilde\varphi^-|{\rm d}\Hausdorff= m[u]\,.
	\end{equation*}
	Now by \cite{DM:80}
there exists a sequence 
$(\varphi_k)\subset C^1(\Omega)$  converging to $\varphi$ in $L^1$ and such that
	\begin{equation*}
		\limsup_{k \to +\infty}\int_\Omega\sqrt{1+|\nabla \varphi_k|^2}\dx\le \int_\Omega \sqrt{1+|\nabla \tilde \varphi|^2}\dx + |D^c\tilde\varphi|(\Omega) + \int_{S_{\tilde\varphi}}|\tilde\varphi^+-\tilde\varphi^-|{\rm d}\Hausdorff\,.
	\end{equation*}
	Therefore the maps
$u_k:=e^{i\varphi_k}\in C^1(\Omega;\mathbb{S}^1)$ 
converge to $u$ in $L^1$ and satisfy \eqref{eq:upper-bound}.
\end{proof}

We conclude the paper by providing a phase-field approximation of Ambrosio-Tortorelli type of $\mathcal{A}^{\mathbb{S}^1}(u,\Omega)$. Let $	\domATpiccolo$ be as in \eqref{eq:two_domains}.
For $\var,h>0$ consider the family of functionals $\mathcal F^{\mathbb{S}^1}_{\var,h}\colon L^1(\Omega;\mathbb{S}^1)\times L^1(\Omega)\to [0,+\infty]$ given by 
\begin{equation}\label{def:F-eps}
	\mathcal F^{\mathbb{S}^1}_{\var,h}(u,v):=\int_\Omega v^2\sqrt{1+|\nabla u|^2}\dx+ \int_\Omega\left(\var|\nabla v|^2
+ \frac{(1-v)^2}{4h\var}
\right)\dx\,,
\end{equation}
when $(u,v)\in{\widehat
\domATpiccolo}$ 
and $+\infty$ otherwise. 
With a little abuse of notation,
we denote by 
$$
\mathcal A^{\mathbb S^1}
\colon L^1(\Omega;\mathbb{S}^1)\times L^1(\Omega)\to [0,+\infty]
$$
 the 
following functional:  
$\mathcal A^{\mathbb S^1}(u,v) = 
\mathcal A^{\mathbb S^1}(u)$ if $u \in BV(\Om)$ and
$v=1$ almost everywhere in $\Omega$, and 
$\mathcal A^{\mathbb S^1}(u,v) = +\infty$
if either 
$u\notin BV(\Omega;\mathbb{S}^1)$ or $v$ is not almost evereywhere
equal to $1$ in $\Omega$.

\begin{theorem}\label{thm:approx-area}Let $\Omega\subset\R^n$ be a bounded connected and simply connected open set with Lipschitz boundary.
Let  $h_\var:=h(\var)$ be positive, with 
$$\lim_{\var\to0}\frac \var{h_\var}=0
 \quad {\rm  and}\quad
 \lim_{\var \to 0^+}
h_\var=0.
$$
Then $\mathcal F^{\mathbb{S}^1}_{\var,h_\var}$ $\Gamma$-converges to 
$\mathcal{A}^{\mathbb{S}^1}$ in the $L^1$-topology.
\end{theorem}
\begin{proof}
	{\textit{Step 1: Lower bound.}} Let $(u_\var,v_\var)\to (u,v)$ in $L^1(\Omega;\mathbb{S}^1)\times L^1(\Omega)$. We need to prove that 
	\begin{equation}\label{eq:liminf}
		\liminf_{\var\to 0^+} \mathcal F^{\mathbb{S}^1}_\var(u_\var,v_\var)\ge 
\mathcal{A}^{\mathbb{S}^1}(u,{v})\,.
	\end{equation}
	Assume, without loss of generality, that 
	\begin{equation*}
		\sup_\var \mathcal F^{\mathbb{S}^1}_\var(u_\var,v_\var)<+\infty\,.
	\end{equation*}
	Let $h>0$ be fixed. Then for $\var$ 
small enough we 
have $h\ge h_\var$, so that 
	\begin{equation*}
		\liminf_{\var\to 0^+}\mathcal F^{\mathbb{S}^1}_{\var,h_\var}(u_\var,v_\var)\ge\liminf_{\var\to 0^+} \mathcal F^{\mathbb{S}^1}_{\var,h}(u_\var,v_\var)\,.
	\end{equation*}
	Moreover from Theorem \ref{thm2}  and \eqref{def:g} we 
have that 
	\begin{equation*}
		\liminf_{\var\to 0^+} \mathcal F^{\mathbb{S}^1}_{\var,h}(u_\var,v_\var)\ge \int_\Omega\sqrt{1+|\nabla u|^2}\dx+|D^cu|(\Omega)+ m_{g_h}
[u]\,,
	\end{equation*}
	where 
	\begin{equation*}
		m_{g_h}[u]=\inf\left\{
		\int_{S_{\varphi}}
g_h(|\varphi^+-\varphi^-|)
		{\rm d}\Hausdorff\colon \varphi\in GBV(\Omega)\,,\ e^{i\varphi}=u\text{ a.e. in }\Omega
		\right\},
	\end{equation*}
with 
$$
g_h(t) := \frac{t}{1+ht}, \qquad t \geq 0.
$$
	Hence
	\begin{equation*}
		\liminf_{\var\to 0^+} \mathcal F^{\mathbb{S}^1}_{\var,h_\var}(u_\var,v_\var)\ge \int_\Omega\sqrt{1+|\nabla u|^2}\dx+|D^cu|(\Omega)+ m_{g_h}[u]\,.
	\end{equation*}
By the arbitrariness of $h>0$, we obtain
\[
\liminf_{\var \to 0^+} \mathcal{F}_{\var, h_\var}^{\mathbb{S}^1}(u_\var, v_\var) \ge \int_{\Omega} \sqrt{1+|\nabla u|^2} \, \dx + |D^c u|(\Omega) + \sup_{h>0} m_{g_h}[u].
\]
Now, we claim that 
\begin{equation}\label{eq:wecla}
\sup_{h>0} m_{g_h}[u] = \lim_{h \to 0^+} m_{g_h}[u] = m[u].
\end{equation}
Indeed, since $g_h(t) \le t$ for all $t \ge 0$, we immediately have $m_{g_h}[u] \le m[u]$ for all $h > 0$, 
hence $\limsup_{h \to 0^+} m_{g_h}[u] \le m[u]$. 

To prove the converse inequality, for every $h > 0$ let $\varphi_h \in GBV(\Omega)$ be a jump-minimizing lifting for $m_{g_h}[u]$,
which exists by Corollary~\ref{cor:existence-minimizer}, so that
\[
m_{g_h}[u] = \int_{S_{\varphi_h}} \frac{|\varphi_h^+ - \varphi_h^-|}{1 + h|\varphi_h^+ - \varphi_h^-|} \, {\rm d}\mathcal{H}^{n-1}.
\]
Fix $h_0 > 0$. For any $h \in (0, h_0)$ we have $g_h(t) \ge g_{h_0}(t)$ for all $t \ge 0$, which yields
\[
\sup_{h \in (0, h_0)} \mathcal{H}^{n-1}\Big(
S_{\varphi_h}^{\pi}
\Big) \le (1 + h_0 \pi) \sup_{h \in (0, h_0)} m_{g_h}[u] \le (1 + h_0 \pi) m[u] < +\infty,
\]
where $S_{\varphi_h}^{\pi}$ is defined in \eqref{eq:Ssigma}.
Thus, we can apply Theorem~\ref{thm:compactness} to the family $(\varphi_h)_{h \in (0, h_0)}$. For a sequence $h_k \to 0^+$, $(\varphi_{h_k})$ converges locally modulo $2\pi$ 
to a lifting $\varphi_0 \in GBV(\Omega)$ of $u$ in $\Om$. 
By the lower semicontinuity result in 
Corollary~\ref{cor:existence-minimizer}-(ii) and the Fatou's lemma, we infer
\[
\liminf_{k \to +\infty} m_{h_k}[u] = \liminf_{k \to +\infty} \int_{S_{\varphi_{h_k}}} \frac{|\varphi_{h_k}^+ - \varphi_{h_k}^-|}{1 + h_k |\varphi_{h_k}^+ - \varphi_{h_k}^-|} \, {\rm d}\mathcal{H}^{n-1} \ge \int_{S_{\varphi_0}} |\varphi_0^+ - \varphi_0^-| \, {\rm d}\mathcal{H}^{n-1} \ge m[u].
\]
This shows claim 
\eqref{eq:wecla}
and concludes the proof of the lower bound.
	\noindent
	
	{\textit{Step 2: Upper bound.}}  
	 Let $u \in BV(\Omega;\mathbb{S}^1)$. We 
need to construct a 
sequence $((u_\var, v_\var))_{\var > 0} \subset \mathcal{D}_{\mathbb{S}^1}$ converging to $(u,1)$ in $L^1(\Omega;\mathbb{S}^1) \times L^1(\Omega)$ such that
	\[
	\limsup_{\var \to 0^+} \mathcal{F}_{\var, h_\var}^{\mathbb{S}^1}(u_\var, v_\var) \le 
\mathcal{A}^{\mathbb{S}^1}(u,1).
	\]
	{
By definition of the relaxed area $\mathcal{A}^{\mathbb{S}^1}(u, \Omega)$, 
there exists a sequence $(u_\var) \subset C^1({\Omega}; \mathbb{S}^1)$ such that $u_\var \to u$ in $L^1(\Omega;\mathbb{S}^1)$ and
\begin{equation}\label{eq:rec_area}
	\lim_{\var \to 0^+} \int_{\Omega} \sqrt{1+|\nabla u_\var|^2} \dx = \mathcal{A}^{\mathbb{S}^1}(u,\Omega).
\end{equation}

Since $u_\var \in C^1({\Omega}; \mathbb{S}^1)$ 
we have $u_\var \in W_{\text{{\rm loc}}}^{1,2}(\Omega; \mathbb{S}^1)$.  
}
%
%
	Setting $v_\var \equiv 1$ on $\Omega$, we have $(u_\var, v_\var) \in \widehat\domATpiccolo$  for all $\var > 0$, and $(u_\var, v_\var) \to (u,1)$ in $L^1(\Omega;\mathbb{S}^1) \times L^1(\Omega)$. 
	Moreover, since $v_\var \equiv 1$, the phase-field terms vanish identically, yielding
	\[
	\mathcal{F}_{\var, h_\var}^{\mathbb{S}^1}(u_\var, v_\var) = \int_{\Omega} \sqrt{1+|\nabla u_\var|^2} \dx.
	\]
Invoking \eqref{eq:rec_area}, we obtain
	\[
	\limsup_{\var \to 0^+} \mathcal{F}_{\var, h_\var}^{\mathbb{S}^1}(u_\var, v_\var) = \lim_{\var \to 0^+} \int_{\Omega} \sqrt{1+|\nabla u_\var|^2} \dx = \mathcal{A}^{\mathbb{S}^1}(u,\Omega),
	\]
	which completes the proof.
\end{proof}

\subsection*{Data availability statement} No new data was created or analysed in this work.

\subsection*{Conflict of interest statement} The authors have no conflict of interest to declare that are relevant to 
the content of this article.

\subsection*{Acknowledgements}
R. Marziani has received funding from the European Union's Horizon research and innovation program under the Marie Skłodowska-Curie agreement No 101150549. 
The research that led to the present paper was partially supported by
 the INdAM-GNAMPA project 2025 CUP E5324001950001. 
The first author benefits from the support of the GNAMPA of INdAM.
The first and third authors also acknowledge the partial financial support of the PRIN project 2022PJ9EFL "Geometric Measure Theory: Structure of Singular
Measures, Regularity Theory and Applications in the Calculus of Variations'', PNRR Italia Domani, funded
by the European Union via the program NextGenerationEU (Missione 4
Componente 2) CUP:E53D23005860006. Views
and opinions expressed are however those of the authors only and do not necessarily reflect those
of the European Union or the European Research Executive Agency. Neither the European Union
nor the granting authority can be held responsible for them.

\end{document}